\documentclass[10pt]{amsart}

\usepackage{amsfonts}
\usepackage{amssymb}
\usepackage{float, graphicx}
\usepackage{amsmath, amsthm, latexsym}
\usepackage[all]{xy}
\usepackage{hyperref}
\usepackage[margin=1.5in]{geometry}
\usepackage{color}

\def \c{\mathbb{C}}
\def \z{\mathbb{Z}}
\def \r{\mathbb{R}}
\def \n{\mathbb{N}}
\def \p{\mathbb{P}}

\def \A{\mathcal{A}}

\def \X{\mathcal{X}}

\def \V{\mathcal{V}}

\def \vol{\textup{vol}}

\def \Re{\textup{Re}}
\def \Im{\textup{Im}}
\def \supp{\textup{supp}}
\def \conv{\textup{conv}}
\def \Sym{\textup{Sym}}
\def \Lie{\textup{Lie}}
\def \ord{\textup{ord}}

\theoremstyle{plain}
\newtheorem{Th}{Theorem}[section]
\newtheorem{Lem}[Th]{Lemma}
\newtheorem{Prop}[Th]{Proposition}
\newtheorem{Cor}[Th]{Corollary}

\theoremstyle{definition}
\newtheorem{Ex}[Th]{Example}
\newtheorem{Def}[Th]{Definition}
\newtheorem{Rem}[Th]{Remark}

\begin{document}

\title{Toric degenerations and symplectic geometry of projective varieties}
\author{Kiumars Kaveh}
\address{Department of Mathematics, University of Pittsburgh,
Pittsburgh, PA, USA.}
\email{kaveh@pitt.edu}

\begin{abstract}
Let $X$ be an $n$-dimensional smooth complex projective variety embedded in $\c\p^{N}$.
We construct a smooth family $\X$ over $\c$ with an embedding in $\c\p^{N} \times \c$ whose generic fiber is $X$
and the special fiber is the torus $(\c^*)^n$ sitting in $\c\p^{N}$ via a monomial embedding.
We use this to show that if $\omega$ is an integral K\"ahler form on $X$ then for any $\epsilon > 0$ there is an open subset 
$U_\epsilon \subset X$ such that $\vol(X \setminus U_\epsilon) < \epsilon$ and $U_\epsilon$ is symplectomorphic to $(\c^*)^n$ equipped 
with a (rational) toric K\"ahler form. As an application we obtain lower bounds for the Gromov width of $(X, \omega)$ in terms of its associated Newton-Okounkov bodies. 
We also show that if $\omega$ lies in the class $c_1(L)$ of a very ample line bundle $L$ then 
$(X, \omega)$ has a full symplectic packing with $d$ equal balls where $d$ is the degree of $(X, L)$. 
\end{abstract}


\thanks{The author is partially supported by a
Simons Foundation Collaboration Grants for Mathematicians and a National Science Foundation Grant.}

\date{\today}
\maketitle


\tableofcontents

\section{Introduction}
Let $X$ be a smooth complex projective variety of dimension $n$ embedded in a projective space $\c\p^{N}$.
We construct a smooth family $\pi: \X \to \c$ together with an embedding into $\c\p^{N} \times \c$ such that
the general fiber of the family is $X$ and the special fiber is the algebraic torus $(\c^*)^n$ embedded in $\c\p^{N}$ via
a monomial embedding (Section \ref{sec-Kahler}).
\footnote{Throughout the paper $\c^*$ denotes $\c \setminus \{0\}$, the multiplicative group of nonzero complex numbers.} Notice that the general fiber of the family is 
$X$ and hence a projective variety, while the special fiber is $(\c^*)^n$ and not projective. In fact if we take the closure $\overline{\X}$ of the family in $\c\p^N \times \c$, the general 
fiber is still $X$ but the special fiber may become reducible while at least one of its irreducible components is an $n$-dimensional toric variety. 
The construction of the family $\X$ is a generalization of the deformation to the normal cone in algebraic geometry. 
The construction of the embedding of $\X$ in $\c\p^N \times \c$ depends on the choice of a $\z^n$-valued valuation on the field of rational functions on $X$. 

We use the family $\X$ and its embedding in $\c\p^N \times \c$ 
to obtain results about the symplectic geometry of $X$. Specifically we prove the following: let $\omega$ be a K\"ahler form on $X$ which is integral (i.e. its cohomology class lies in $H^2(X, \z)$). Then for any $\epsilon > 0$ there exists an open subset $U \subset X$ (in the usual classical topology) such that $\vol(X \setminus U) < \epsilon$ and $(U, \omega)$ is 
symplectomorphic to $(\c^*)^n$ equipped with a toric K\"ahler form (Theorem \ref{th-U-large-intro} below and Theorem \ref{th-main2}). 

From this we obtain results about the Gromov width and symplectic ball packing problem for $X$ 
(Section \ref{sec-GW-packing}). In particular, when $\omega$ lies in the class $c_1(L)$ of a very ample line bundle $L$ 
we conclude that $(X, \omega)$ has a full symplectic packing with $d$ equal balls where $d$ is the degree of $L$ (i.e. the self-intersection of the divisor class of $L$).


Let us explain the above more precisely. Fix a finite set $\A = \{\beta_1, \ldots, \beta_r\} \subset \z^n$ and a point $c = (c_1, \ldots, c_r) \in (\c^*)^r$. 
We assume that the set of differences of elements in $\A$ generates the lattice $\z^n$ and hence the orbit map:
\begin{equation} \label{equ-intro-psi}
\psi_{\A, c}: u \mapsto (u^{\beta_1}c_1: \cdots : u^{\beta_r}c_r),
\end{equation}
is an isomorphism of varieties from $(\c^*)^n$ to its image $O_{\A, c} \subset \c\p^{r-1}$. Here $u =(u_1, \ldots, u_n) \in (\c^*)^n$ and $u^{\alpha}$ is shorthand for $u_1^{a_1} \cdots u_n^{a_n}$, 
where $\alpha = (a_1, \ldots, a_n)$. The closure of $O_{\A, c}$ is a (not necessarily normal) projective toric variety.
The map $\psi_{\A, c}$ also induces a K\"ahler form on $(\c^*)^n$ as follows:
Consider the standard Hermitian product on $\c^r$. Let $\Omega$ be the associated Fubini-Study K\"ahler form on $\c\p^{r-1}$ and let $\omega_{\A, c}$ be
the pull-back of $\Omega$ to $(\c^*)^n$ under the map $\psi_{\A, c}$. 
The symplectic manifold $((\c^*)^n, \omega_{\A, c})$ is a Hamiltonian space with 
respect to the natural action of the compact torus $(S^1)^n$ on $(\c^*)^n$ by multiplication. The image of its moment map is the interior of the convex hull of $\A$.

Now let $X \subset \c\p^{r-1}$ be a smooth projective variety embedded in some projective space $\c\p^{r-1}$. 
We construct a complex manifold $\X$ together with a holomorphic function $\pi: \X \to \c$, as well as a an 
embedding $\X \hookrightarrow \c\p^{r-1} \times \c$ such that:
\begin{itemize}
\item[(a)] The family $\X$ is trivial over $\c^*$ i.e. $\pi^{-1}(\c^*) \cong X \times \c^*$. In particular for each $t \neq 0$ we have $X_t := \pi^{-1}(t)$ is biholomorphic to $X$.
Moreover, $X_1 \hookrightarrow \c\p^{r-1} \times \{1\}$ coincides with the original embedding $X \hookrightarrow \c\p^{r-1}$.
\item[(b)] The fiber $X_0 = \pi^{-1}(0)$ is the algebraic torus $(\c^*)^n$ embedded in $\c\p^{r-1} \times \{0\}$ via a monomial map $\psi_{\A, c}$ for some finite set $\A \subset \z^n$ and 
$c \in (\c^*)^r$ as above.
\item[(c)] The map $\pi: \X \to \c$ has no critical points, i.e. $d\pi$ is nonzero at every point in $\X$.
\end{itemize}

Next consider a smooth projective variety $X$ equipped with a very ample line bundle $L$. 
The line bundle $L$ gives rise to the Kodaira embedding 
$X \hookrightarrow \p(H^0(X, L)^*)$. Let $\omega$ be a K\"ahler form in the class $c_1(L)$. 
We note that by Moser's trick any two K\"ahler forms in $c_1(L)$ are symplectomorphic. In particular $\omega$ is
symplectomorphic to the pull-back of a Fubini-Study K\"ahler form on the projective space $\p(H^0(X, L)^*)$ to $X$.

Using the family $\X$ above we prove the following (Theorem \ref{th-main}):
\begin{Th} \label{th-U-torus-intro}
There exists an open subset $U \subset X$ (in the usual classical topology) such that $(U, \omega)$ is 
symplectomorphic to $((\c^*)^n, \omega_{\A, c})$, for some $\A \subset \z^n$ and $c \in (\c^*)^r$ as above (i.e. $\A$ is a finite subset such that the 
differences of elements in $\A$ generate the lattice $\z^n$).
\end{Th}


Consider a K\"ahler form on $(\c^*)^n$ of the form $\frac{1}{m}\omega_{\A, c}$, where $m$ is a positive integer. We call such a form a {\it rational toric K\"ahler form}. 
The following is our main result about the symplectic geometry of smooth projective varieties. It states that we can enlarge the open subset $U$ in Theorem \ref{th-U-torus-intro}
as much as we wish provided that we consider rational toric K\"ahler forms on $(\c^*)^n$.

Let $X$ be a smooth projective variety and let $\omega$ be a K\"ahler form on $X$ which is integral (i.e. its cohomology class lies in $H^2(X, \z)$). 
We recall that by the Lefschetz theorem on $(1, 1)$-classes any integral K\"ahler form is in $c_1(L)$ for an ample line bundle $L$.
\begin{Th} \label{th-U-large-intro}
For any $\epsilon > 0$ we can find an open subset $U \subset X$ such that 
$\vol(X \setminus U) < \epsilon$ and $(U, \omega)$ is symplectomotphic to $(\c^*)^n$ equipped with a rational toric K\"ahler form.
\end{Th}

Roughly speaking, this result claims that {\it in symplectic category and over arbitrarily large open subsets, any smooth projective variety looks like a 
toric variety equipped with a toric K\"ahler form.}



In Section \ref{sec-GW-packing} we discuss some applications of Theorem \ref{th-U-large-intro} to symplectic topology.
Before explaining these applications here let us briefly recall some important concepts from symplectic topology.

The famous {\it non-squeezing theorem} of Gromov states that a ball $B_{2n}(r)$ of radius $r$, in the symplectic vector space $\r^{2n}$ with the standard structure, 
cannot be embedded symplectically into a cylinder $B_2(R) \times \r^{2n-2}$ of radius $R$, unless $r \leq R$. 
Hence, the non-squeezing theorem tells us that, while symplectic maps are volume-preserving, it is much more restrictive for a map to be 
symplectic than it is to be volume-preserving (\cite{Gromov}).
Motivated by the Gromov non-squeezing theorem, one introduces the notion of the {\it Gromov width}
which is an important invariant of a symplectic manifold. It is defined as the supremum of capacities of symplectic balls that can be embedded in the manifold.  
More precisely, let $(X, \omega)$ be a symplectic manifold of (real) dimension $2n$. Consider the ball of radius $r$:
$$B_{2n}(r)= \{ z=(z_1, \ldots, z_n) \in \c^n \mid \sum_{i=1}^n |z_i|^2 < r^2 \},$$ 
equipped with the standard symplectic form $\omega_{\textup{st}} = \sum_i dx_i \wedge dy_i$. The quantity $\pi r^2$ is usually referred to as 
the {\it capacity} of the ball $B_{2n}(r)$. The Gromov width of $(X, \omega)$ 
is the supremum of the set $\{\pi r^2 \mid B_{2n}(r) \textup{ can be symplectically embedded in } (X, \omega)\}$.

Another related concept is that of a symplectic packing. As above let $(X, \omega)$ be a symplectic manifold of (real) dimension $2n$. 
Fix an integer $N \geq 1$. A {\it symplectic packing of $(M, \omega)$ by $N$ equal balls of radius $r$}  is a symplectic embedding:
$$\underbrace{B_{2n}(r) \amalg \cdots \amalg B_{2n}(r)}_{N \textup{ times}} \to M,$$
of a disjoint union of $N$ balls of equal radius $r$ into $M$.
The Darboux theorem ensures that, given $N$, such a packing always exists provided that the radius $r$ is small enough. But when we increase the radius $r$ there will be
obstructions for existence of a packing. An obvious obstruction is the total volume of the packing, since a symplectic embedding preserves the volume. In symplectic packing theory
one investigates other obstructions for symplectic packings beyond the obvious volume obstruction. One defines:
$$v_N(X, \omega) = \sup_r \frac{N \vol(B_{2n}(r)}{\vol(M, \omega)},$$
where the supremum is over all the radii $r$ for which there is a symplectic packing of $(X, \omega)$ with $N$ balls of radius $r$. If $v_N(X, \omega) = 1$ one 
says that $(X, \omega)$ has a {\it full symplectic packing} by $N$ equal balls. If $v_N(X, \omega) < 1$ one says that there is a {\it packing obstruction} (see \cite{Biran}). 

Now let $X$ be a smooth complex projective variety with an ample line bundle $L$ and a K\"ahler form $\omega$ in $c_1(L)$. 
In Section \ref{sec-GW-packing} we show the following. 
\begin{itemize}
\item We give a lower bound for the Gromov width of $(X, \omega)$ in terms of an associated convex body $\Delta \subset \r^n$, namely its Newton-Okounkov body (see Section 
\ref{sec-enlarge} for the definition of a Newton-Okounkov body). More precisely 
the Gromov width is at least {the supremum of the sizes of simplices that lie in the interior of a Newton-Okounkov body of $(X, L)$} 
(see Definition \ref{def-size-simplex} and Corollary \ref{cor-GW-NO-body}). 
In particular, this implies that when $L$ is very ample the Gromov width of $(X, \omega)$ is at least $1$. 
\item Moreover, we show that when $L$ is very ample the symplectic 
manifold $(X, \omega)$ has a full symplectic packing by $d$ equal balls, where $d$ is the degree of the line bundle $L$ (i.e. the self-intersection number of its divisor class).
\end{itemize}
Paul Biran has asked whether the Gromov width of a compact symplectic manifold is at least $1$ if the symplectic form is integral. As said above, 
our approach in particular confirms this for smooth projective varieties and very ample line bundles. The fact that the Gromov width in the case of a very ample line bundle is at 
least $1$ can be obtained by other methods involving the notion of Seshadri constant (see \cite[Theorem 5.1.22]{Lazarsfeld-book}, \cite{McDuff}). 

As suggested to the author by Robert Lazarsfeld, the construction of the family $\X$ and its embedding in this paper, may be useful in approaching the 
Ein-Lazarsfeld conjecture. It states that the Seshadri constant of an ample line bundle on a smooth projective variety is at least $1$ (\cite[Conjecture 5.2.4]{Lazarsfeld-book}). 
This conjecture implies Biran's conjecture for smooth projective varieties and ample line bundles.

We would like to point out that after the first version of the present paper was posted on arXiv.org, following the methods and ideas here, 
Witt Nystr\"om proved a K\"ahler version of the statement about symplectic ball embeddings (\cite{WittNystrom}).

Also after the first version of the paper appeared we learned about the work of Atsushi Ito on Seshadri constants (\cite{Ito}). 
Our lower bound on the Gromov width in terms of the size of the largest simplex is closely related to the result in \cite{Ito} in which the author 
obtains similar bounds for the Seshadri constants using different methods.
Finally we would like to mention the recent interesting papers \cite{KL-surface} and \cite{KL-positivity}. In there the authors prove criteria for
positivity of line bundles in terms of their associated Newton-Okounkov bodies. In particular, in the case of surfaces they prove a lower bound for the 
Seshdari constant (\cite[Proposition 4.7]{KL-surface}). 


As mentioned before, the construction of the family $\X$ and its embedding in $\c\p^{r-1} \times \c$ in this paper is a generalization of the deformation to the normal cone 
and is related to the Rees algebra construction in commutative algebra. Nevertheless in this paper we do not explicitly use Rees algebras. 
Interested reader can look at \cite[Sections 6.5 and 15.8]{Eisenbud} for basic material about Rees algebras as well as \cite[Section 2]{Teissier}) and 
\cite{Anderson} for construction of flat families from valuations.
It is also worthwhile to point out the relationship with Gr\"obner degenerations. In the Gr\"obner basis theory one degenerates an ideal $I$ in a polynomial ring to its
initial ideal $\textup{\bf in}(I)$ which then gives a flat degeneration of the scheme $X$ defined by $I$ to the scheme $X_0$ defined by $\textup{\bf in}(I)$. The problem is that even when $I$ is 
a prime ideal, the ideal $\textup{\bf in}(I)$ is usually not radical. This makes it harder to obtain information about the geometry of $X$ from this degeneration. On the other hand, in our construction 
instead of degenerating the defining relations of $X$ (i.e. the ideal), we degenerate a set of generators for the (homogeneous) coordinate ring of $X$ to their initial terms. We show that with suitable choices, these initial terms are monomials and define an embedding of the torus $(\c^*)^n$ into the projective space.

One of the ingredients in the proof of Theorem \ref{th-U-torus-intro} is the notion of a {\it gradient-Hamiltonian vector field} due
to Ruan (see \cite[Section 3.1]{Ruan} as well as Section \ref{sec-grad-Hamiltonian} in this paper). 
The author learned about this notion in the interesting paper \cite{NNU} where the gradient-Hamiltonian flow is applied to the Gelfand-Zetlin toric degeneration of the flag variety. 
The symplectomorphism between the torus and an open subset $U$ of $X$ in Theorem \ref{th-U-torus-intro} is given by 
the flow of the gradient-Hamiltonian vector field of $\pi: \X \to \c$ with respect to the K\"ahler structure coming from the embedding $\X \hookrightarrow \c\p^{r-1} \times \c$. 

The toric degenerations usually considered in algebraic geometry literature, have slightly different properties than our degeneration satisfying properties (a)-(c) above
(see for example \cite{Lakshmibai, AB, Anderson, Kaveh-crystal}). \footnote{The term toric degeneration is also used for certain families of varieties whose special fiber is 
a union of toric varieties, see \cite{Gross}.} Let $X$ be an irreducible projective variety. 
What is usually understood by a toric degeneration of $X$, 
is a family of varieties $\pi: \X \to \c$, where $\X$ is a variety and $\pi$ is a morphism, which has the following properties:
\begin{itemize}
\item[(1)] The family is trivial over $\c^*$, i.e. $\pi^{-1}(\c^*)$ is isomorphic to $X \times \c^*$. In particular, 
each fiber $X_t = \pi^{-1}(t)$, $t \neq 0$, is isomorphic to $X$
\item[(2)] The special fiber $X_0 = \pi^{-1}(0)$ is a projective toric variety.
\item[(3)] All the fibers are irreducible and reduced (as schemes).
\item[(4)] The family $\X$ is flat over $\c$ (this is in fact implied by the previous conditions).
\end{itemize}

Suppose the family $\X$ is sitting in $\c\p^N \times \c$, for some projective space $\c\p^N$ and such that $\pi$ is the restriction of the projection on the second factor.
The conditions (3) and (4) then give the following: the Hilbert polynomial of $X_t$, as a subvariety of $\c\p^N$, is independent of $t$ 
(see \cite[Chapter III, Theorem 9.9]{Hartshorne}). Let us consider the product K\"ahler form on $\c\p^N \times \c$ where $\c\p^N$ is equipped 
with the standard Fubini-Study K\"ahler form and $\c$ with the standard K\"ahler form $\frac{i}{2} dz \wedge d\bar{z}$. 
Equip (the smooth locus of) $\X$ with a K\"ahler structure induced from $\c\p^N \times \c$. The conditions above in particular tell us that the symplectic volume of 
all the fibers $X_t$, including the special fiber $X_0$, are the same. 

In \cite{Anderson}, motivated by the previous toric degeneration result of \cite{AB}, 
Anderson constructs a large class of flat toric degenerations of projective varieties in connection with the theory of Newton-Okounkov bodies. 
The construction depends on a choice of a $\z^n$-valued valuation $v$ on the field of rational functions on the variety. A crucial assumption in this construction is that an associated semigroup of 
lattice points is finitely generated (the value semigroup of $v$).
In \cite{HK}, given a toric degeneration $\X$ for a smooth projective variety $X$ as above, 
the authors construct a completely integrable system on $X$ by pulling back the toric integrable system on the toric variety $X_0$ 
(to do this one needs a compatible K\"ahler structure on the family $\X$). 
Since in general the special fiber $X_0$ is non-smooth, the pull-back integrable system is defined only on an open subset (in the usual classical topology) of $X$. 
The fact that the toric variety $X_0$ has the same symplectic volume as $X$ implies that the open subset on $X$ where the integrable system is defined is dense. 
A main result in \cite{HK} is that the integrable system extends continuously to the whole $X$.
The image of this integrable system is the Newton-Okounkov body associated to $X$ and the valuation $v$ (used in building the toric degeneration).

Our construction in this paper works in a much more general setting. 
Here we do not require the technical and restrictive 
assumption of ``finite generation of the value semigroup'' needed in \cite{Anderson, HK}. 
We note that instead of a projective toric variety, our special fiber $X_0$ is just the torus $(\c^*)^n$, and moreover it can have a smaller symplectic volume than that of the 
original variety $X$. 
That is why the toric open subset $U$ in Theorem \ref{th-U-torus-intro} may not be dense in $X$. Nevertheless if the special fiber $X_0 = (\c^*)^n$ happens to have the 
same symplectic volume as $X$ then $U$ is dense in $X$. These are illustrated in Section \ref{sec-example}, Examples \ref{ex-1} and \ref{ex-2}.

We expect that our results can be generalized to non-smooth projective varieties as well (here we use smoothness to guarantee that the gradient-Hamiltonian flow
is defined for all $t \geq 0$).

We believe that the results of this paper open new doors to study symplectic geometry of projective varieties and in particular the notions 
of Gromov width and symplectic packings, and more applications
beyond those in Section \ref{sec-GW-packing} would appear.
In particular, as pointed out by Dusa McDuff, the results of the paper may help to give new bounds for the Gromov width as well as symplectic packings in several important classes of 
examples such as abelian surfaces and $K3$ surfaces. Also they might be useful in other packing problems such as 
for the symplectic polydisc embeddings (even in real dim $4$ and for square polydiscs $B^2(a)\times B^2(a)$ not much is known e.g. for general tori $T^2(1)\times T^2(2)$ or 
for K3 surfaces).

\medskip
\noindent{\bf Acknowledgements:}
The author would like to thank Yael Karshon and Milena Pabiniak for suggesting to use toric degeneration techniques to study lower bounds on Gromov width and symplectic packings. 
Special thanks to Megumi Harada for much helpful discussions and patiently checking some details. 
Also the author is grateful to Paul Biran, Dusa McDuff, Alex K\"uronya, Behrang Noohi, Dave Anderson, Mahdi Majidi-Zolbanin, Chris Manon, Kristin Shaw, Sue Tolman, Atsushi Ito, Bernard 
Teissier and Askold Khovanskii for very useful conversations and correspondence. 

\section{Preliminaries on toric K\"ahler structures on $(\c^*)^n$} \label{sec-prelim}
Fix a finite set $\A = \{\beta_1, \ldots, \beta_r\} \subset \z^n$. Consider the linear action of the algebraic torus $(\c^*)^n$ on $\c\p^{r-1}$ by:
$$u \cdot (z_1: \cdots :z_r) = (u^{\beta_1}z_1: \cdots : u^{\beta_r}z_r).$$
Here $u =(u_1, \ldots, u_n) \in (\c^*)^n$ and $u^{\alpha}$ is shorthand for $u_1^{a_1} \cdots u_n^{a_n}$, where $\alpha = (a_1, \ldots, a_n)$.
Take a point $c=(c_1: \ldots: c_r) \in \c\p^{r-1}$ with $c_i \neq 0$ for all $i$. Let $O_{\A, c} = (\c^*)^n \cdot c$ denote the orbit of this point, i.e.:
$$O_{\A, c} = \{ (u^{\beta_1}c_1 : \cdots : u^{\beta_r}c_r) \mid u \in (\c^*)^n \}.$$
Let us assume that the set of differences $\{ \alpha - \beta \mid \alpha, \beta \in \A\}$ generates $\z^n$. 
One observes that under this assumption the orbit map $\psi_{\A, c}: u \mapsto u \cdot c$ is an isomorphism of algebraic varieties
between $(\c^*)^n$ and $O_{\A, c}$. The closure of $O_{\A, c}$ is a projective toric variety $X_{\A, c}$. In general, this variety is not normal and hence not smooth 
(see \cite[Section 2.1]{CLS}).

Consider the standard Hermitian product on $\c^r$. Let $\Omega$ be the associated Fubini-Study K\"ahler form on $\c\p^{r-1}$ and $\omega_{\A, c}$ denote 
the pull-back of $\Omega$ to $(\c^*)^n$ under the map $\psi_{\A, c}$. 
 
The following is a reformulation of the Bernstein-Kushnirenko theorem (see \cite{Kush, Bern}). This formulation as well as the symplectic proof of the Bernstein-Kushnirenko theorem 
is due to A. G. Khovanskii (see \cite[Section 5]{Atiyah}).
\begin{Th}[Bernstein-Kushnirenko] \label{th-BK}
The degree of $X_{\A, c}$, as a subvariety of the projective space $\c\p^{r-1}$, is equal to $n!$ times the 
Euclidean volume of the convex polytope $\Delta = \conv(\A)$ (this in turn is equal to $n!$ times the symplectic volume of $X_{\A, c}$).
\end{Th}


The K\"ahler potential of $\omega_{\A, c}$ is given by:
$$ K_{\A, c}(u)  = \log(\sum_{j=1}^r |c_j| |u|^{2\beta_j}), \quad u \in (\c^*)^n.$$

The form $\omega_{\A, c}$ is invariant under the natural action of the compact torus $(S^1)^n$ on 
$(\c^*)^n$ by multiplication. Also $((\c^*)^n, \omega_{\A, c})$ is a Hamiltonian $(S^1)^n$-space. The image of 
the moment map of this Hamiltonian space is the interior of the convex hull of $\A$.

\section{A family of complex manifolds over $\c$} \label{sec-family}
Let $X$ be a complex manifold of dimension $n$. Consider the product manifold $X \times \c^*$. We regard it as a trivial family 
with the fiber map $\pi: X \times \c^* \to \c^*$, the projection on the $\c^*$ factor.
In this section we explain a way to attach a copy of $(\c^*)^n$ to this family as the fiber over $0$ to obtain a larger complex manifold $\X$ as well as a holomorphic extension of $\pi$ 
to $\X$:
$$\xymatrix{
X \times \c^*~ \ar[rd]_{\pi} \ar@{^{(}->}[r] & \X \ar[d]^{\pi} \\ & \c\\
}$$

All the fibers of the family $\X$ except the fiber over zero, namely $(\c^*)^n$, are biholomorphic to $X$. 
The construction of the family $\X$ depends on the choice of a point $p \in X$ and a system of coordinates at this point.

The construction of $\X$ is a weighted version of the well-known deformation to the normal cone from algebraic geometry (see \cite[Section 5.1]{Fulton}). 
One takes an $n$-dimensional variety $X$ and a subvariety $Y \subset X$. The deformation to the normal cone is a family of varieties (over $\c$) which deforms $X$ to the normal cone of 
$Y$ in $X$. In fact here we deal only with the case where $X$ is smooth and $Y$ is a single point $p$. In this case, the normal cone to $p$ is just the tangent space 
$T_pX \cong \c^n$. One constructs the deformation $\tilde{\X}$ as follows: define $\tilde{\X}$ to be the blow-up of $X \times \c$ at the point $(p, 0)$. The projection on the 
second factor $X \times \c \to \c$ gives a regular map $\pi: \tilde{\X} \to \c$. For $t \neq 0$ the fiber $\pi^{-1}(t)$ is naturally isomorphic to $X$. One shows that the fiber 
$\pi^{-1}(0)$ is the union of two hypersurfaces $Z_1$ and $Z_2$ where $Z_1$ is the blow-up of $X$ at $p$ and $Z_2$ is a projective space isomorphic to $\c\p^n$ which 
contains the affine space $T_pX = \c^n$ as an affine chart. \footnote{In the algebraic geometry literature one usually considers $X \times \c\p^1$ (instead of $X \times \c$) 
and blows it up at the point $(p, \infty)$ (instead of $(p, 0)$).}  Now if one removes $Z_1$ as well as all the coordinate planes in $Z_2 \cong \c\p^n$ what we obtain is a 
family $\pi: \X \to \c$ which satisfies the above, namely the generic fiber $\pi^{-1}(t)$, $t \neq 0$ is $X$ and the special fiber $\pi^{-1}(0)$ is $(\c^*)^n$. 
Below we discuss a weighted version of this construction for complex manifolds.

Let $X$ be a complex manifold of dimension $n$. Let $p \in X$ and let $u_1, \ldots, u_n$ be local coordinates at $p$ (in particular $u_1(p) = \cdots = u_n(p) = 0$).
Let $U \subset X$ be an open subset (in the Euclidean topology of $X$) such that the map:
$$\phi = (u_1, \ldots, u_n): U \to \c^n,$$
is a biholomorphism between $U$ and its image $\phi(U)$ which is an open subset of $\c^n$ containing $0$. 

Let us fix a vector $\gamma = (\gamma_1, \ldots, \gamma_n) \in \n^n$ (here $\n$ denotes the set of positive integers).
Consider the map $\tilde{\phi}: U \times \c^* \to \c^{n+1}$ defined by:
$$\tilde{\phi}(x, t) = (t^{-\gamma_1}u_1(x), \ldots, t^{-\gamma_n}u_n(x), t).$$

\begin{Prop} \label{prop-tilde-phi}
(1) The image $\tilde{\phi}(U \times \c^*)$ is open in $\c^{n+1}$ (in the Euclidean topology). (2) $\tilde{\phi}$ is a biholomorphism between $U \times \c^*$ and its image.
\end{Prop}
\begin{proof}
Consider the map $\Psi: \c^n \times \c^* \to \c^n \times \c^*$ given by $$\Psi(u_1, \ldots, u_n, t)=(t^{-\gamma_1}u_1, \ldots, t^{-\gamma_n}u_n, t).$$
The inverse of $\Psi$ is given by $(u_1, \ldots, u_n, t) \mapsto (t^{\gamma_1}u_1, \ldots, t^{\gamma_n}u_n, t)$ which is clearly holomorphic.
Thus $\Psi$ is a biholomorphism. In parcitular $\Psi$ is an open map. This readily implies that $\tilde{\phi}(U \times \c^*) = \Psi(\phi(U) \times \c^*)$
is open in $\c^n \times \c^*$ and hence in $\c^{n+1}$. 
To prove (2) note that $\tilde{\phi}$ is the composition of the map $(x, t) \mapsto (\phi(x), t)$ with $\Psi$. Since both of these maps are 
biholomorphisms so is $\tilde{\phi}$. This finishes the proof.
\end{proof}


Now consider the set $\tilde{U} \subset \c^{n+1}$ defined by: 
\begin{equation} \label{equ-tilde-U}
\tilde{U} = \tilde{\phi}(U \times \c^*) \cup ((\c^*)^n \times \{0\}).
\end{equation}
\begin{Prop}
$\tilde{U}$ is an open subset of $\c^{n+1}$ (in the Euclidean topology).
\end{Prop}
\begin{proof}
By Proposition \ref{prop-tilde-phi} we know that $\tilde{\phi}(U \times \c^*)$ is open. So we just need to show that any point in $(\c^*)^n \times \{0\}$ is an 
interior point of $\tilde{U}$. Take $(u_1, \ldots, u_n, 0) \in (\c^*)^n \times \{0\}$. Let $\delta > 0$ be such that the poly-disc $\{(w_1, \ldots, w_n) \mid |w_i| < \delta,~ i = 1, \ldots, n\}$ 
is contained in $\phi(U)$. Take $\epsilon > 0$ such that:
$$\epsilon < \min \{ (\frac{\delta}{(|u_i| + 1)})^{1/\gamma_i} \mid i = 1, \ldots, n\}.$$
Thus if $(u'_1, \ldots, u'_n, s) \in (\c^*)^n \times \c^*$ is such that $|s| < \epsilon$ and $|u'_i - u_i| < 1$ for $i=1, \ldots, n$ then:
$$|s^{\gamma_i} u'_i| = |s|^{\gamma_i} |u'_i| \leq |\epsilon|^{\gamma_i} (|u_i|+1) < \delta.$$
So $(s^{\gamma_1}u'_1, \ldots, s^{\gamma_n}u'_n) \in \phi(U)$ which implies that: $$(u'_1, \ldots, u'_n, s) = 
(s^{-\gamma_1}(s^{\gamma_1}u'_1), \ldots, s^{-\gamma_n}(s^{\gamma_n}u'_n), s) \in \tilde{\phi}(U \times \c^*).$$ This finishes the proof. 
\end{proof}

Now we consider the following set:
$$\X = ((X \times \c^*) \amalg \tilde{U}) / \sim,$$
that is, $\X$ is the disjoint union of the complex manifolds $X \times \c^*$ and $\tilde{U}$ quotient by the equivalence relation $\sim$ defined as follows:
Let $(x, t) \in X \times \c^*$ and $(y_1, \ldots, y_n, t) \in \tilde{\phi}(U \times \c^*)$ we say that 
$(x, t) \sim (y_1, \ldots, y_n, t)$ if $(x, t)$ lies in $U \times \c^*$ and $\tilde{\phi}(x, t) = (y_1, \ldots, y_n, t)$.
In other words, $\X$ is the complex manifold obtained by gluing $X \times \c^*$ and $\tilde{U}$ along their open subsets $U \times \c^*$ and 
$\tilde{\phi}(U \times \c^*)$ using the biholomorphism $\tilde{\phi}$.

We also define a map $\pi: \X \to \c$ as follows: for $(x, t) \in X \times \c^*$ let $\pi(x, t) = t$ and for $(\tilde{u}, t) \in \tilde{U} \subset \c^{n+1}$ let $\pi(\tilde{u}, t) = t$. 
It is clear that $\pi$ is well-defined. For each $t \in \c$ we denote 
the fiber $\pi^{-1}(t)$ by $X_t$. It is also clear from the definition of $\X$ and $\pi$ that $X_t \cong X$ for $t \neq 0$ and $X_0$ is the algebraic torus $(\c^*)^n$.

Since $\tilde{\phi}$ is a bihomolorphism we conclude:
\begin{Prop} \label{prop-family}
We have the following:
\begin{itemize}
\item[(1)] The set $\X$ has the structure of a complex manifold. 
\item[(2)] The map $\pi: \X \to \c$ is holomorphic.
\item[(3)] The map $\pi$ has no critical points.
\end{itemize}
\end{Prop}
\begin{proof}
(1) Firstly, it is easy to see that with quotient topology
$\X = ((X \times \c^*) \sqcup \tilde{U}) / \sim$ is a Hausdorff space.
Next we recall that under the identification $\sim$ the open subset $U \times \c^* \subset X \times \c^*$ is identified with $\tilde{\phi}(U \times \c^*)$. The latter is an open subset of 
$\c^n \times \c^*$ and moreover $\tilde{\phi}$ is a biholomorphism between these two open sets by Proposition \ref{prop-tilde-phi}. Thus $\X$ is obtained from the complex manifold 
$X \times \c^*$ by gluing another complex manifold $\tilde{U} \subset \c^{n+1}$ along the open subset $\tilde{\phi}(U \times \c^*)$. This gives $\X$ the structure of a complex manifold.
To prove (2) and (3) we note that in the open subset $X \times \c^*$ the map $\pi$ is given by $(x, t) \mapsto t$ which is clearly holomorphic and has no critical points (its differential is never zero).
Also in the other open set $\tilde{U} \subset \c^{n+1}$ the map $\pi$ is defined by $(u_1, \ldots, u_n, t) \mapsto t$ which is also holomorphic with no critical points. This proves the proposition.  
\end{proof}

Figure \ref{fig-tilde-phi} illustrates the map $\tilde{\phi}$ and the family $\pi: \X \to \c$.
\begin{figure}[h]
\includegraphics[width=9cm]{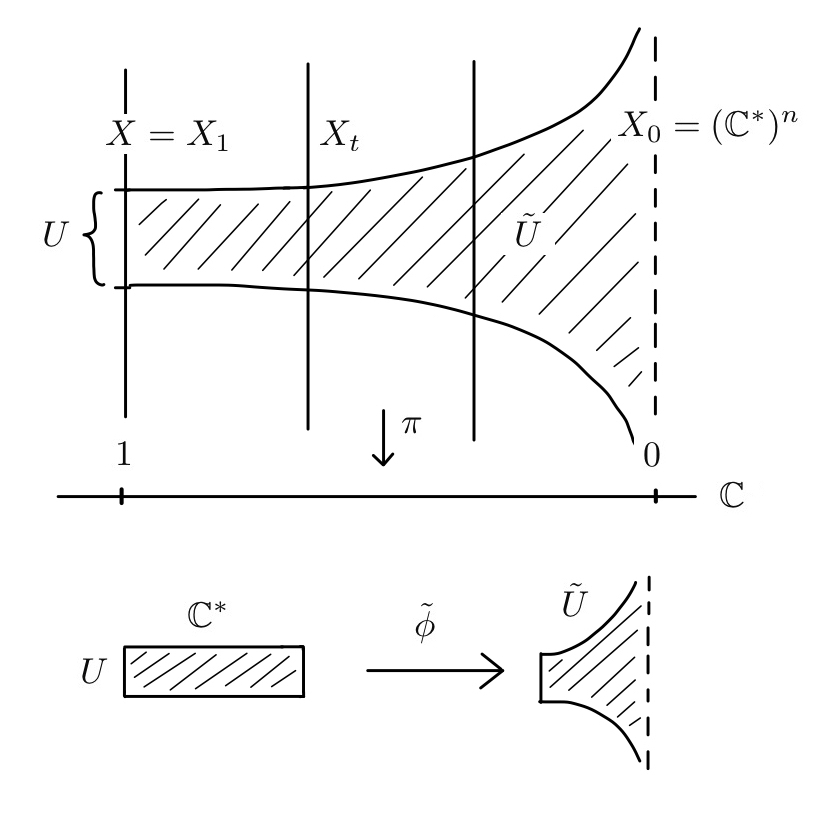} 
\caption{Illustration of the map $\tilde{\phi}$, the open set $\tilde{U}$ and the family $\X$ over $\c^*$} 
\label{fig-tilde-phi} 
\end{figure}

\section{Analytic extension of meromorphic functions to the special fiber} \label{sec-power-series}
As before let $X$ be a complex manifold and $p$ a point on $X$. Also let $u_1, \ldots, u_n$ denote local coordinates in an open neighborhood $U$ of $p$.
Let $f$ be a meromorphic function on $X$ which is holomorphic at $p$ and let
\begin{equation} \label{equ-f}
f = \sum_{\alpha \in \z_{\geq 0}^n} c_{\alpha} u^{\alpha},
\end{equation}
be the power series expansion of $f$ in $u = (u_1, \ldots, u_n)$. Here as usual we have used the multi-index notation $u^\alpha$ to denote $u_1^{\alpha_1} \cdots u_n^{\alpha_n}$ where
$\alpha = (\alpha_1, \ldots, \alpha_n)$. 
Since $f$ is holomorphic at $p$ we know the following:
\begin{Prop} \label{prop-abs-convergent}
There exists an open neighborhood $V$ (in the Euclidean topology) of the origin such that the power series in \eqref{equ-f} converges absolutely in $V$.
\end{Prop}
Let us take $\delta > 0$ such that the closed poly-disc $\{(u_1, \ldots, u_n) \mid |u_i| \leq \delta,~i=1, \ldots, n \}$ is contained in $V$. In particular, the series:
\begin{equation} \label{equ-delta}
\sum_{\alpha=(\alpha_1, \ldots, \alpha_n) \in \z^n_{\geq 0}} c_\alpha \delta^{\alpha_1 + \cdots+ \alpha_n},
\end{equation}
is absolutely convergent.

Let $$\supp(f) = \{ \alpha \in \z^n_{\geq 0} \mid c_\alpha \neq 0\} \subset \z^n_{\geq 0}$$ be the set of exponents which appear in the power series expansion \eqref{equ-f} of $f$.
As before fix an integer vector $\gamma = (\gamma_1, \ldots, \gamma_n) \in \n^n$. Let $\beta \in \supp(f)$ be a point where the minimum of $\alpha \mapsto \gamma \cdot \alpha$ 
is attained on $\supp(f)$. Note that in general $\beta$ is not unique (although later in Section \ref{sec-valuation} we will choose $\gamma$ carefully so that the minimum is obtained at a 
unique $\beta$). Define
the meromorphic function $\tilde{f}$ on $X \times \c^*$ by:
$$\tilde{f} = t^{-\gamma \cdot \beta} f.$$

For $i=1, \ldots, n$, let $\tilde{u}_i$ denote the holomorphic function:
$$\tilde{u}_i = t^{-\gamma_i} u_i,$$
on $U \times \c^*$. By Proposition \ref{prop-tilde-phi} the $\tilde{u}_i$ form a coordinate system on an open neighborhood of the special fiber $X_0 = (\c^*)^n \times \{0\}$.

We have:
\begin{align} 
\tilde{f} &= t^{-\gamma \cdot \beta} \sum_{\alpha \in \z^n_{\geq 0}} c_{\alpha} u^\alpha, \\
&= \sum_{\alpha \in \z^n_{\geq 0}} c_{\alpha} t^{-\gamma \cdot \beta} t^{\gamma \cdot \alpha} \tilde{u}^\alpha, \\
\label{equ-tilde-f}
&= \sum_{\alpha \in \z^n_{\geq 0}} c_{\alpha} \tilde{u}^\alpha (t^{\gamma \cdot \alpha - \gamma \cdot \beta}).
\end{align}
Note that the last line is a power series in $\tilde{u} = (\tilde{u}_1, \ldots, \tilde{u}_n)$ and $t$.


\begin{Th} \label{th-tilde-f-convergent}
For each $(\tilde{u}, 0) \in (\c^*)^n \times \{0\}$, the power series of $\tilde{f}$ is absolutely convergent in an open neighborhood of $(\tilde{u}, 0)$ in $\c^{n+1}$.
Thus $\tilde{f}$ is holomorphic at $(\tilde{u}, 0)$.
\end{Th}
\begin{proof}
Suppose we have $\epsilon > 0$ such that for all $\tau \in \c$ with $|\tau| < \epsilon$, and for all $\alpha \in \supp(f)$ with $\gamma \cdot \alpha \neq \gamma \cdot \beta$, and $i=1, \ldots, n$ 
we have:
\begin{equation} \label{equ-tau}
|\tau|^{1 - {\frac{\gamma \cdot \beta}{\gamma \cdot \alpha}}} \leq (\frac{\delta}{|\tilde{u}_i| + 1})^{\frac{1}{\gamma_i}}.
\end{equation}
Take any $\tilde{w}=(\tilde{w}_1, \ldots, \tilde{w}_n) \in \c^n$ in the open poly-disc of radius $1$ around $\tilde{u}$. That is, $|\tilde{w}_i - \tilde{u}_i| <1$ for $i=1, \ldots, n$. 
Then $|\tilde{w}_i| < |\tilde{u}_i| + 1$, for all $i$. 
If \eqref{equ-tau} holds then for all $i = 1, \ldots, n$ we have:
$$|\tilde{w}_i| |\tau|^{\gamma_i (1 - \frac{\gamma \cdot \beta}{\gamma \cdot \alpha})} \leq \delta.$$
Let us write $\alpha = (\alpha_1, \ldots, \alpha_n)$. We then obtain:
\begin{align*} \label{equ-tau2}
\prod_{i=1}^n (|\tilde{w}_i|^{\alpha_i} |\tau|^{\gamma_i \alpha_i(1 - \frac{\gamma \cdot \beta}{\gamma \cdot \alpha})}) \leq& \delta^{\alpha_1 + \cdots + \alpha_n} \\
(\prod_{i=1}^n |\tilde{w}_i|^{\alpha_i}) |\tau|^{\gamma \cdot \alpha - \gamma \cdot \beta} \leq& \delta^{\alpha_1 + \cdots + \alpha_n}.
\end{align*}
Recall that by the choice of $\delta$ (see the paragraph after Proposition \ref{prop-abs-convergent}) the series \eqref{equ-delta} is absolutely convergent. 
Note also that the the set of $\alpha \in \z^n_{\geq 0}$ with $\gamma \cdot \alpha = \gamma \cdot \beta$ is finite. So \eqref{equ-tau} holds for all $\alpha \in \supp(f)$ 
except possibly for a finite number. Thus by the comparison test we conclude that the series:
$$\sum_{\alpha \in \z^n_{\geq 0}} c_{\alpha} \tilde{w}^\alpha (t^{\gamma \cdot \alpha - \gamma \cdot \beta}).$$
is also absolutely convergent for $(\tilde{w}, \tau) \in \c^{n+1}$ where $|\tau| < \epsilon$ and $|\tilde{w}_i - \tilde{u}_i| < 1$ for all $i=1, \ldots, n$.

It remains to show that there exists $\epsilon > 0$ satisfying \eqref{equ-tau}.  We first note the following obvious facts: 
(1) The set $\{ \gamma \cdot \alpha \mid \alpha \in \z^n_{\geq 0}\}$ is discrete. 
(2) The set $\{ \alpha \in \z^n_{\geq 0} \mid \gamma \cdot \alpha = \gamma \cdot \beta \}$ is finite.  

We claim that there is a constant $c > 0$ such that for all $\alpha \in \supp(f)$ with $\gamma \cdot \alpha \neq \gamma \cdot \beta$ we have:
\begin{equation} \label{equ-c}
0 < c \leq 1 - \frac{\gamma \cdot \beta}{\gamma \cdot \alpha}.
\end{equation}
If such a $c$ does not exist then there is a sequence $\{\alpha_i\}_{i \in \n}$ consisting of distinct points in $\supp(f)$ such that $\lim_{i \to \infty} 
(1 - \frac{\gamma \cdot \beta}{\gamma \cdot \alpha_i}) = 0$. This implies that $\lim_{i \to \infty} \gamma \cdot \alpha_i = \gamma \cdot \beta$ which contradicts 
the facts (1) and (2) above. 

Let $c > 0$ be as above. Then by \eqref{equ-c}, for any $\tau \in \c$ with $|\tau| < 1$ we have:
\begin{equation} \label{equ-tau3}
|\tau|^{1 - \frac{\gamma \cdot \beta}{\gamma \cdot \alpha}} \leq |\tau|^c. 
\end{equation}
(This is because the real function $x \mapsto \tau^x$ is decreasing for $0 \leq \tau < 1$.) Now pick $\epsilon > 0$ such that:
$$\epsilon < \min\{ 1, (\frac{\delta}{|\tilde{u}_1|+1})^{\frac{1}{c\gamma_1}}, \ldots, (\frac{\delta}{|\tilde{u}_n|+1})^{\frac{1}{c\gamma_n}}\}.$$
Then if $|\tau| < \epsilon$ we have $|\tau| < 1$ and moreover for $i=1, \ldots, n$:
$$|\tau|^{1 - \frac{\gamma \cdot \beta}{\gamma \cdot \alpha}} \leq |\tau|^c < (\frac{\delta}{{|\tilde{u}_i|}+1})^{\frac{1}{\gamma_i}}.$$
This finishes the proof of the theorem.
\end{proof}

\begin{Rem} \label{rem-analytic}
Alternatively, as pointed out to the author by Megumi Harada, we can see that $\tilde{f}(\tilde{u}, t)$ is holomorphic at any $(\tilde{u}, 0) \in (\c^*)^n$ as follows.
Consider $f$ as a power series in $\tilde{u}$ and $t$:
$$f(\tilde{u}, t) = \sum_{\alpha \in \z_{\geq 0}^n} c_\alpha t^{\gamma \cdot \alpha} \tilde{u}^\alpha.$$
Since $f$ is absolutely convergent for $u$ in some neighborhood $V$ of the origin, and from the proof of Proposition \ref{prop-tilde-phi}, we can see that 
$f(\tilde{u}, t)$ is absolutely convergent in a neighborhood of $(\tilde{u}, 0)$. Also, by the choice of $\beta$, the power series $f(\tilde{u}, t)$ is divisible 
by $t^{\gamma \cdot \beta}$. It thus follows that the power series $\tilde{f} = t^{-\gamma \cdot \beta}f$ is also absolutely convergent in a neighborhood of $(\tilde{u}, 0)$. 
This proves the claim.
\end{Rem}

\section{Choice of an orthonormal basis} \label{sec-valuation}
Now let $X$ be an $n$-dimensional smooth complex projective variety. Consider an embedding of $X$ in a projective space $\p(V)$ where $V$ is a finite dimensional 
vector space. Let $E$ be the associated linear system on $X$, i.e. the image of the restriction map $V^* = H^0(\p(V), \mathcal{O}(1)) \to H^0(X, \mathcal{O}(1)_{|X})$.
Without loss of generality we assume that the restriction map is one-to-one and we identify $E$ with the dual vector space $V^*$.

Take a point $p \in X$. Let $(u_1, \ldots, u_n)$ be a system of coordinates in a neighborhood of $p$, that is, $u_1, \ldots, u_n$ are regular functions at $p$ with 
$u_1(p) = \cdots = u_n(p) = 0$ such that $du_1, \ldots, du_n$ are linearly independent at $p$. 

To this coordinate system we associate a $\z^n$-valued valuation $v$ on the field of rational functions $\c(X)$ as follows: Let $f \in \mathcal{O}_{X, p}$ 
be a regular function at $p$. As before, $f$ can be represented by a power series $\sum_{\alpha \in \z_{\geq 0}^n} c_\alpha u^\alpha$, where $u = (u_1, \ldots, u_n)$. 
Fix a total order $>$ on the lattice $\z^n$ which respects addition, e.g. a lexicographic order. Let us define the value of the valuation $v$ at $f$ to be:
\begin{equation} \label{equ-valuation}
v(f) = \min\{ \alpha \mid c_\alpha \neq 0\}.
\end{equation}
Here the minimum is taken with respect to the total oder $>$. As usual in algebra one extends $v$ to the whole field $\c(X)$ by defining $v(f/g) = v(f) - v(g)$ for $f, g \in \mathcal{O}_{X, p}$.
It is straightforward to verify that $v$ satisfies the defining axioms of a valuation.

Fix a section $\tau \in E$ such that $\tau(p) \neq 0$. The choice of the section $\tau$ is about trivializing the associated line bundle in a neighborhood of the point $p$. 
We then obtain a map from $E \setminus \{0\}$ to $\z^n$ given by:
$$\sigma \in E \mapsto v(\sigma/\tau).$$ 
Let $\A$ denote the image of this map in $\z^n$, that is:
\begin{equation} \label{equ-def-A}
\A = \{ v(\sigma / \tau) \mid 0 \neq \sigma \in E \}.
\end{equation}
One shows, using the properties of a valuation, that the set $\A$ is finite and in fact $|\A| = r = \dim(E)$ (for example see \cite[Proposition 2.6]{KKh-Annals}).

Now fix a Hermitian product $H$ on the dual vector space $E^*$.
One can construct a basis
$\{\eta_1, \ldots, \eta_r\}$ of $E$ such that its dual basis is an orthonormal basis for $E^*$ and moreover:
$$\{v(\eta_1/\tau), \ldots, v(\eta_r / \tau)\} = \A.$$ 
(The proof is a Gram-Schmidt orthonormalization argument, see \cite[Lemma 3.23]{HK} for a proof.)

\begin{Rem} \label{rem-basis-dual-basis}
If we identify $E$ with its dual $E^*$, via the Hermitian product $H$, then $\{\eta_1, \ldots, \eta_r\}$ is indeed an orthonormal basis with respect to the induced 
Hermitian product.
\end{Rem}

\begin{Rem} \label{rem-eta_j}
In fact, for our purposes in Section \ref{sec-Kahler} it suffices if $\{\eta_1^*, \ldots, \eta_r^*\}$ is an orthonormal basis and
the set of differences of elements in $\{v(\eta_1 / \tau), \ldots, v(\eta_r / \tau)\}$ generates the lattice $\z^n$.
\end{Rem}

Let us fix such an orthornormal basis $\{\eta_1, \ldots, \eta_r\}$. For each $j=1, \ldots, r$ put:
$$f_j = \frac{\eta_j}{\tau} \quad \textup{and} \quad \beta_j = v(\eta_j / \tau).$$
Also let
\begin{equation} \label{equ-f_i}
f_j = \sum_{\alpha \in \z^n_{\geq 0}} c_{\alpha, j} u^\alpha,
\end{equation}
be the power series representation of $f_j$ in $u=(u_1, \ldots, u_n)$.

One can show the following:
\begin{Prop} \label{prop-perturb-gamma}
We can find a vector $\gamma \in \n^n$ such that for each $j=1, \ldots, r$, the vector
$\beta_j$ is the unique point in $\supp(f_j)$ where the minimum of $\{ \gamma \cdot \alpha \mid \alpha \in \supp(f_j) \}$ is attained.  
\end{Prop}
\begin{proof}
For each $j$ let $\Delta_j = \conv(\supp(f_j)) \subset \r_{\geq 0}^n$. Choose $\gamma \in \n^n$ such that for each $j$ the hyperplane 
$\{ x \in \r^n \mid \gamma \cdot x = \gamma \cdot \beta_j\}$ intersects the convex polyhedron $\Delta_j$ only at $\beta_j$.
\end{proof}
 
As in Section \ref{sec-power-series} put:
\begin{equation} \label{equ-tilde-f_i}
\tilde{f}_j = t^{-\gamma \cdot \beta_j} f_j = \sum_{\alpha \in \z^n_{\geq 0}} c_{\alpha, j} \tilde{u}^\alpha (t^{\gamma \cdot \alpha - \gamma \cdot \beta_j}).
\end{equation}
 

We thus obtain the following:
\begin{Prop} \label{prop-const-term-monomial}
For each $j$, the constant term with respect to $t$ in the power series $\tilde{f}_j(\tilde{u}, t)$ (i.e. the term with no $t$) is a monomial in $\tilde{u}$, namely $c_{j} \tilde{u}^{\beta_j}$.
\end{Prop}

Note that in the notation of the equation \eqref{equ-tilde-f_i} we have $c_j = c_{\beta_j, j}$. 

From Proposition \ref{th-tilde-f-convergent} we have the following:
\begin{Cor}
For each $(\tilde{u}, 0) \in X_0 = (\c^*)^n \times \{0\}$, all the power series $\tilde{f}_j$, $j=1, \ldots, r$, are absolutely convergent in a neighborhood of $(\tilde{u}, 0)$ in $\c^{n+1}$.
In particular all the $\tilde{f}_j$ are holomorphic functions in an open subset in $\c^{n+1}$ containing $(\c^*)^n \times \{0\}$.
\end{Cor}
 

We will be interested in the valuations $v$ such that the following holds: the set of differences of elements in $\A$, namely:
$$\{ \alpha - \beta \mid \alpha, \beta \in \z\},$$
generates the lattice $\z^n$. Under this condition, in Section \ref{sec-Kahler} we will construct an embedding of the family $\X$ in a product $\c\p^{r-1} \times \c$. 
We will use this to define a K\"ahler form on the whole family.

\begin{Rem} \label{rem-diff-A-generates}
Suppose $p \in X$ and a section $\tau \in E$ with $\tau(p) \neq 0$ is given.
Let us show that one can always choose a coordinate system $(u_1, \ldots, u_n)$ at $p$ such that 
the corresponding set $\A$ of initial exponents, as in \eqref{equ-def-A}, has the property that the differences of elements in $\A$ generates $\z^n$.

Let $\sigma_1, \ldots, \sigma_n \in E$ be sections such that: 
\begin{itemize}
\item[(1)] For each $i=1, \ldots, n$, we have $\sigma_i(p) = 0$. 
\item[(2)] The hypersurfaces $D_i = \{x \in X \mid \sigma_i(x) = 0 \}$, $i=1, \ldots, n$, are smooth at $p$. 
\item[(3)] The hypersurfaces $D_1, \ldots, D_n$ are transverse at $p$.
\end{itemize}
For each $i = 1, \ldots, n$ define:
$$u_i = {\frac{\sigma_i}{\tau}}.$$
Then the $u_i$ are rational functions on $X$ that are regular at $p$, in other words $u_i \in \mathcal{O}_{X, p}$.
From the properties (1)-(3) above and in particular the transversality of the $D_i$ we conclude that 
$(u_1, \ldots, u_n)$ form a system of coordinates in a neighborhood of the point $p$. 
We note that for each $i = 1, \ldots, n$, $v(\sigma_i / \tau) = v(u_i) = e_i$ where $e_i$ is the $i$-th standard basis element in $\r^n$.  
Moreover $v(\tau / \tau) = v(1) = 0$. Thus
the finite subset $\A$ contains the standard basis as well as the origin. In particular, the differences of the elements in $\A$ generate the lattice $\z^n$.
\end{Rem}
 
\section{Embedding the family in a projective space and K\"ahler structure on the family} \label{sec-Kahler}
As before let $X$ be a smooth complex projective variety together with a linear system $E$ 
(i.e. a finite dimensional vector space of holomorphic sections of a line bundle on $X$)
such that $E$ gives rise to a Kodaira embedding $X \hookrightarrow \p(E^*)$ where $E^*$ is the dual space of $E$.

Fix a Hermitian product $H$ on $E^*$. The choice of $H$ gives a Fubini-Study K\"ahler form $\Omega_H$ on $\p(E^*)$.
We let $\omega = \omega_H$ denote the pull-back of $\Omega_H$ to $X$ under the embedding $X \hookrightarrow \p(E^*)$. 

Using the notation in Section \ref{sec-valuation} let $f_1, \ldots, f_r$ be the rational function on $X$ corresponding to the orthonormal basis $\{\eta_1, \ldots, \eta_r\}$
and a section $\tau \in E$ with $\tau(p) \neq 0$ (the choice of section $\tau$ corresponds to a trivialization of the associated line bundle in a neighborhood of the point $p$).
Moreover, we assume that the differences of elements in $\A$ generates the lattice $\z^n$.
 
Let us identify $E$ with $\c^r$ using the basis $\{\eta_1, \ldots, \eta_r\}$. 
We will denote the coordinates on $\c^r$ (respectively homogeneous coordinates on $\c\p^{r-1}$) by $(z_1, \ldots, z_r)$ (respectively $(z_1: \cdots :z_r)$). 

We define the map $F: \X \to \c\p^{r-1} \times \c$ as follows. Let $\tilde{x} \in \X$ then:
$$F(\tilde{x}) = ((\tilde{f}_1(\tilde{x}) : \cdots : \tilde{f}_r(\tilde{x})), \pi(\tilde{x})).$$
In other words, if $(x, t) \in X \times \c^* \subset \X$ then:
$$F(x, t) = ((t^{-\gamma_1}f_1(x) : \cdots : t^{-\gamma_r}f_r(x)), t),$$
and if $(\tilde{u}, t) \in \tilde{U} \subset \c^{n+1}$ then:
$$F(\tilde{u}, t) = ((\tilde{f}_1(\tilde{u}, t) : \cdots : \tilde{f}_r(\tilde{u}, t)), t).$$
Recall that $\tilde{U}$ is the open subset in $\X$ containing the zero fiber $X_0 = (\c^*)^n$ (see Section \ref{sec-family}) and the $\tilde{f}_j$ are the 
analytic extensions of $t^{-\gamma_j}f_j$ to the whole family $\X$ (see Section \ref{sec-power-series}).

\begin{Rem}
Note that $F$ is not defined at the locus in $X \times \c^*$ where $\tau = 0$. It is easy to check, as we now show, that the function $F$ is independent of the 
choice of the section $\tau$ and hence $F$ can be extended to a holomorphic map on the whole family $\X$.
Choose another section $\tau' \in E$ with $\tau'(p) \neq 0$ 
and put $\tilde{f}'_j = t^{-\gamma \cdot \beta_j} (\eta_j / \tau')$. Then we have $\tilde{f}'_j (x, t) = (\tau(x) / \tau'(x)) \tilde{f}_j(x, t)$ for each $j$ and 
any $x \in X$ with $\tau(x) \neq 0$, $\tau'(x) \neq 0$. It follows that:
$$(\tilde{f}_1(x, t): \cdots : \tilde{f}_r(x, t)) = (\tilde{f}'_1(x,t) : \cdots : \tilde{f}'_r(x, t)).$$
This implies that $F$ is independent of the choice of $\tau$ as claimed. 
\end{Rem}

\begin{Th} \label{th-F-immersion}
The map $F$ is an immersion, i.e. its derivative at every point has maximum rank $n+1$. 
\end{Th}
\begin{proof}
First we show that the derivative of $F$ has maximum rank at every point in $X \times \c^*$. Take $\tilde{x} = (x, t) \in X \times \c^*$.
Without loss of generality we assume $\tau(x) \neq 0$ and hence $F(x, t)$ is given by:
$$F(x, t) = ((t^{-\gamma \cdot \beta_1}f_1(x) : \cdots : t^{-\gamma \cdot \beta_r}f_r(x)), t).$$
We note that the map:
$$(x, t) \mapsto ((f_1(x) : \cdots : f_r(x)), t),$$ is an embedding of $X \times \c^*$ into $\c\p^{r-1} \times \c^*$ simply because $\{\eta_1, \ldots, \eta_r\}$ is a basis 
for $V$. Also the map:
$$((z_1 : \cdots : z_r), t) \mapsto ((t^{-\gamma \cdot \beta_1}z_1 : \cdots : t^{-\gamma \cdot \beta_r}z_r ), t),$$
is a biholomorphism from $\c\p^{r-1} \times \c^*$ to itself. It follows that the map $F$ which is the composition of these two maps has maximum rank $n+1$ at $(x, t)$.
It remains to show that the map $F$ has maximum rank at every point of the special fiber $X_0 = (\c^*)^n$. Take $\tilde{x} = (\tilde{u}, 0) \in X_0$. 
Recall that in a neighborhood of $X_0$ each $\tilde{f}_j$ is given by a convergent power series in $\tilde{u}$ and $t$ of the form:
$$\tilde{f}_j(\tilde{u}, t) = c_j \tilde{u}^{\beta_j} + \textup{ higher terms containing } t, \quad c_j \neq 0.$$ 
Cleary for each $j=1, \ldots, r$, we have $\tilde{f}_j(\tilde{u}, 0) = c_j \tilde{u}^{\beta_j} \neq 0$. Thus in the affine chart $z_r \neq 0$ the map $F$ is given by:
$$F(\tilde{u}, t) = ((\frac{\tilde{f}_1(\tilde{u}, t)}{\tilde{f}_r(\tilde{u}, t)}, \ldots, \frac{\tilde{f}_{r-1}(\tilde{u},t)}{\tilde{f}_r(\tilde{u},t)}), t).$$
Note that each $\tilde{f}_j / \tilde{f}_r$ is a convergent power series in a neighborhood of $(\tilde{u}, 0)$ of the form:
$$\frac{\tilde{f}_j(\tilde{u}, t)}{\tilde{f}_r(\tilde{u}, t)} = d_j \tilde{u}^{\lambda_j} + \textup{ higher terms containing } t,$$
where $d_j = c_j / c_r$ and $\lambda_j = \beta_j - \beta_r$.  We thus can compute the Jacobian of $F$ at $(\tilde{u}, 0)$ to be:
\begin{equation} \label{equ-Jac-F}
JF(\tilde{u}, 0) = \left[ 
\begin{matrix}
\lambda_{11} d_{1} \tilde{u}^{\lambda_1 - e_1} & \cdots & \lambda_{1n} d_{1} \tilde{u}^{\lambda_1 - e_n} & * \\
\lambda_{21} d_{2} \tilde{u}^{\lambda_2 - e_1} & \cdots & \lambda_{2n} d_{2} \tilde{u}^{\lambda_2 - e_n} & * \\
& \cdots & & \\
\lambda_{r-1,1} d_{r-1} \tilde{u}^{\lambda_{r-1} - e_1} & \cdots & \lambda_{r-1,n} d_{r-1} \tilde{u}^{\lambda_{r-1} - e_n} & * \\
0 & \cdots & 0 & 1 \\ 
\end{matrix}
\right].
\end{equation}  
Here $\lambda_j = (\lambda_{j1}, \ldots, \lambda_{jn})$, $j=1, \ldots, r-1$, and also $e_i$, $i=1, \ldots, n$, denotes the $i$-th standard basis element in $\r^n$.
We claim that the matrix \eqref{equ-Jac-F} has full rank equal to $n+1$. To prove this note that the rank of a matrix does not change if each row or column is multiplied with a nonzero scalar. 
Multiplying the $i$-th column of the matrix in \eqref{equ-Jac-F} by the $\tilde{u}_i$ and the $j$-th row by $1 / (d_{j} \tilde{u}^{\lambda_j})$ we obtain the matrix:
\begin{equation} \label{equ-Jac-F-2}
\left[ 
\begin{matrix}
\lambda_{11} & \cdots & \lambda_{1n} & * \\
\lambda_{21} & \cdots & \lambda_{2n} & * \\
& \cdots & & \\
\lambda_{r1} & \cdots & \lambda_{rn} & * \\
0 & \cdots & 0 & 1 \\ 
\end{matrix}
\right].
\end{equation}
This matrix has rank $n+1$ because the upper left $n \times r$ submatrix has rank $n$.
This is because by assumption the set $\{\lambda_1, \ldots, \lambda_r\} = \{\beta_1 - \beta_r, \ldots, \beta_{r-1} - \beta_r \}$ 
generates $\z^n$. This proves the claim. 
\end{proof}

\begin{Rem}
One can show that $F$ is one-to-one and hence an embedding. In fact,
the image of $F$ is an algebraic variety.
\end{Rem}

Let $\tilde{\Omega}$ be the K\"ahler form on $\c\p^{r-1} \times \c$ which is 
the product of the standard Fubini-Study K\"ahler form on $\c\p^{r-1}$ and the standard K\"ahler form $\frac{i}{2} dz \wedge d\bar{z}$ on $\c$.
From Theorem \ref{th-F-immersion} it follows that the pull-back $\tilde{\omega} = F^*(\tilde{\Omega})$ is a K\"ahler form on the family $\X$.

For each $t \in \c$ let $\omega_t$ denote the restriction of $\tilde{\omega}$ to the fiber $X_t$. The next statements follow from the definition of $F$ and $\tilde{\omega}$.
\begin{Prop} \label{prop-omega-1-omega-0}
(1) The form $\omega_1$ coincides with the original K\"ahler form $\omega$ on $X \cong X_1$. (2) 
The form $\omega_0$ is the toric K\"ahler form $\omega_{\A, c}$ associated to the finite subset $\A \subset \z^n$ and $c = (c_1, \ldots, c_r) \in (\c^*)^r$.
\end{Prop}
\begin{proof}
(1) The claim is immediate from the definition of $\tilde{\omega}$ and the fact that the dual basis $\{\eta_1^*, \ldots, \eta_r^*\}$ is an orthonormal basis. 
(2) By Proposition \ref{prop-const-term-monomial} the constant term with respect to $t$ of 
each power series $\tilde{f}_j$ is a monomial $c_j\tilde{u}^{\beta_j}$. Hence when $t=0$ the K\"ahler form coincides with $\omega_{\A, c}$.
\end{proof}

\section{Gradient-Hamiltonian flow} \label{sec-grad-Hamiltonian}
Let $\X$ be a K\"ahler manifold with a K\"ahler form $\tilde{\omega}$. 
Let $\pi: \X \to \c$ be a holomorphic map. As usual we let $X_t$ denote the fiber $\pi^{-1}(t)$ and $\omega_t$ denote the restriction 
of $\tilde{\omega}$ to (the smooth locus of) $X_t$. 
Following Ruan \cite[Section 3.1]{Ruan}, we define the
\emph{gradient-Hamiltonian vector field} corresponding to $\pi$ on
$\X$ as follows. Let $\nabla(\Re(\pi))$ denote the gradient
vector field on $\X$ associated to the real part $\Re(\pi)$ with respect to
the K\"ahler metric. Since $\tilde{\omega}$ is K\"ahler and $\pi$ is holomorphic,
the Cauchy-Riemann equations imply that $\nabla(\Re(\pi))$ is related
to the Hamiltonian vector field 
$\xi_{\Im(\pi)}$ of the imaginary part $\Im(\pi)$ with respect to the
K\"ahler (symplectic) form $\tilde{\omega}$ by:  
\begin{equation}\label{gradient Re pi and Hamiltonian Im pi}
\nabla(\Re(\pi)) = - \xi_{\Im(\pi)}.
\end{equation}
Let $W \subset \X$ denote the critical set of $\pi$, i.e. the set on which $d\pi = 0$. Note that by the Cauchy-Riemann relations a point is a critical point of $\pi$ if and only if it is a 
critical point of $\Re(\pi)$.  
The {\it gradient-Hamiltonian vector field} $\V_\pi$, which is defined only on the open set $\X \setminus W$,
is by definition: 
\begin{equation}\label{def-grad-Hamiltonian}
\V_\pi := - \frac{\nabla(\Re(\pi))}{\|\nabla(\Re(\pi))\|^2}.  
\end{equation}
Where defined, $\V_\pi$ is smooth.
For $t \in \r_{\geq 0}$ let $\phi_t$ denote the time-$t$ flow
corresponding to $\V_\pi$. 
Since $\V_\pi$ may not be complete, $\phi_t$ for a given $t$ is not
necessarily defined on all of $\X \setminus W$. 

The following are main properties of the gradient-Hamiltonian vector field. For the sake of completeness we provide proofs.
\begin{Prop} 
\label{prop-grad-Hamiltonian}
Let the notation be as introduced above.
\begin{itemize}
\item[(a)] Suppose $s,t \in \r$ with $s \geq t > 0$. Where defined, the flow $\phi_t$ takes $X_s \cap (\X \setminus W)$ to
$X_{s-t}$. 
\item[(b)] Where defined, the flow $\phi_t$ preserves the symplectic
structures, i.e., if $x \in X_s \cap (\X \setminus W)$ is a point
where $\phi_t(x)$ is defined, then
$\phi_t^*(\omega_{s-t})_{\phi_t(x)} = (\omega_s)_x$. 
\end{itemize}
\end{Prop}
\begin{proof}
The claim (a) immediately follows from the following simple computation:
$$\V(\Re(\pi)) = - \frac{1}{\|\nabla(\Re(\pi))\|^2} \langle \nabla(\Re(\pi)), \nabla(\Re(\pi)) \rangle = -1.$$
(In fact, this is the reason for the normalization factor $1 / \| \nabla(\Re(\pi)) \|^2$ in the definition of the gradient-Hamiltonian vector field.)
To prove (b) it suffices to show that at any point $x \in \X$ we have $\mathcal{L}_\V(\omega)(X, Y) = 0$ where $X$, $Y$ are smooth vector fields which are tangent to 
fibers of $\pi$ and $\mathcal{L}_\V$ is the Lie derivative along $\V$. We use Cartan's magic formula for the Lie derivative and the fact that $d\omega = 0$ to write:
$$\mathcal{L}_\V(\omega)(X, Y) = d \circ \iota_\V(\omega)(X, Y).$$
Next we recall the following coordinate-free formula for the exterior derivative of a $1$-form $\alpha$: 
$$d\alpha(X, Y) = X(\alpha(Y)) - Y(\alpha(X)) - \alpha([X, Y]),$$
where $[X, Y]$ denotes the Lie bracket of vector fields $X$ and $Y$.
Applying this to the $1$-form $\iota_\V(\omega)$ we obtain:
\begin{equation} \label{equ-ext-deriv}
d \circ \iota_\V(\omega)(X, Y) = X (\omega(\V, Y)) - Y(\omega(\V, X)) - \omega(\V, [X, Y]).
\end{equation}
Finally, from the definition of a Hamiltonian vector field, 
one knows that the tangent space to a fiber of $\pi$ at any point
lies in the symplectic orthogonal to the Hamiltonian vector field $\xi_{\Im(\pi)}$ at that point. Since the vector fields $X$, $Y$ and $[X, Y]$ are all tangent to the fibers of
$\pi$ and the vector field $\V$ is parallel to $\xi_{\Im(\pi)}$, 
we conclude that the righthand side of \eqref{equ-ext-deriv} is equal to $0$. This finishes the proof. 
\end{proof}

\section{A toric open subset on the variety} \label{sec-main}
In this section we let $X$ be a smooth $n$-dimensional complex projective variety equipped with a very ample line bundle $L$. We have the Kodaira embedding 
$X \hookrightarrow \p(E^*)$ where $E=H^0(X, L)$. As in Section \ref{sec-Kahler} fix a Hermitian product $H$ on $E^*$. The choice of $H$ gives a Fubini-Study K\"ahler form 
$\Omega_H$ on $\p(E^*)$. We let $\omega$ denote the pull-back of $\Omega_H$ to $X$ under the embedding $X \hookrightarrow \p(E^*)$. The K\"aher form $\omega$ represents the
class of $c_1(L)$ (we recall that by Moser's trick any other K\"ahler form in $c_1(L)$ is symplectomorphic to $\omega$).

In Sections \ref{sec-family} we constructed a complex manifold $\X$ together with a holomorphic function $\pi: \X \to \c$ such that:
\begin{itemize}
\item The family is trivial over $\c^*$ i.e. $\pi^{-1}(\c^*) \cong X \times \c^*$. In particular for each $t \neq 0$ we have $X_t := \pi^{-1}(t)$ is isomorphic to $X$.
\item The fiber $X_0 = \pi^{-1}(0)$ is the algebraic torus $(\c^*)^n$.
\item The map $\pi: \X \to \c$ has no critical points, i.e. $d\pi$ is nonzero at every point in $\X$.
\end{itemize}

Moreover, in Sections \ref{sec-power-series}, \ref{sec-valuation} and \ref{sec-Kahler} we constructed a K\"ahler form $\tilde{\omega}$ on $\X$ such that:
\begin{itemize}
\item The restriction $\omega_1$ of the K\"ahler form $\tilde{\omega}$ to the fiber $X_1$ is the original K\"ahler form $\omega$ on $X \cong X_1$.
\item The restriction $\omega_0$ of the K\"ahler form $\tilde{\omega}$ to the special fiber $X_0 = (\c^*)^n$ is an integral toric K\"ahler form $\omega_{\A, c}$ on $(\c^*)^n$ corresponding to 
a subset $\A = \{\beta_1, \ldots, \beta_r\} \subset \z^n$ and $c = (c_1, \ldots, c_r) \in (\c^*)^r$. 
\end{itemize}

The next theorem is one of the main results of the paper.
\begin{Th} \label{th-main}
With notation as above, there exists an open susbset $U \subset X$ (in the usual classical topology) such that $(U, \omega)$ is symplectomorphic to $((\c^*)^n, \omega_{\A, c})$.
\end{Th}


We need the following lemma.
\begin{Lem} \label{lem-grad-Ham-flow}
The gradient-Hamiltonian flow $\phi_{-t}$ is defined for all $(\tilde{u}, 0) \in X_0 = (\c^*)^n \times \{0\}$ and all $t \geq 0$. 
\end{Lem}
\begin{proof}
The lemma follows from Proposition \ref{prop-grad-Hamiltonian}(a), the fact that $\pi$ has no critical points and that the boundary of the manifold $\X$ lies in the special fiber 
$X_0$ as each general fiber $X_t$, $t \neq 0$, is biholomorphic to the compact manifold $X$.
\end{proof}

\begin{proof}[Proof of Theorem \ref{th-main}]
Consider the gradient-Hamiltonian flow $\phi_{-t}$ of the function $\pi: \X \to \c$. By Lemma \ref{lem-grad-Ham-flow} 
flow $\phi_{-t}$ is defined on all of $X_0 = (\c^*)^n$ and for all $t \geq 0$. In particular $\phi_{-1}: X_0 \to X_1$ the
is a diffeomorphism between $X_0 = (\c^*)^n$ and its image $U = \phi_{-1}(X_0) \subset X_1 \cong X$. The set $U$ is open because the flow of a vector field is an open map. 
Moreover, since the gradient-Hamiltonian flow preserves the symplectic forms, we get
a symplectomorphism between $(X_0, \omega_0)$ and $(U, \omega_1)$. The theorem now follows from Proposition \ref{prop-omega-1-omega-0}.  
\end{proof}

\begin{Rem} \label{rem-U-not-dense}
Note that in general the open subset $U$ is not dense in $X$. The open subset $U$ is dense in $X$ whenever $(X, \omega)$ and $(X_0, \omega_0)$ have the 
same volume. This is the situation appearing in \cite{HK}.
\end{Rem}


\section{Two examples} \label{sec-example}
\begin{Ex} \label{ex-1}
Consider the elliptic curve $$X = \{(x : y : z) \mid y^2z = x^3 + z^3\},$$ in $\c\p^2$. The point at infinity on the elliptic curve is $(0:1:0)$.
In the affine chart $\{z \neq 0\}$ the curve is given by the equation $y^2 = x^3 + 1$.

Consider the standard Hermitian product $H$ on $\c^3$, i.e. the standard basis $\{e_1, e_2, e_3\}$ is an orthonormal basis. 
Equip $X$ with the K\"ahler structure induced from the Fubini-Study K\"ahler form on the projective space $\c\p^2$ corresponding to $H$.
Let $L$ be the restriction of the hyperplane bundle $\mathcal{O}(1)$ to $X$. Then the space of global holomorphic sections $H^0(X, L)$ naturally can be 
identified with the dual vector space $(\c^3)^*$. The coordinate functions $\{x, y, z\}$ are the dual basis to the standard basis $\{e_1, e_2, e_3\}$.
We denote the sections in $H^0(X, L)$ corresponding to these coordinate functions also by $x, y, z$.

Take the point $p = (0 : 1 : 1)$ on $X$. The function $u = x/z$ is a local coordinate in some neighborhood $U$ of the point $p$ on $X$. 
We can take $U$ to be $\{(x : y : 1) \in X \mid |x| < 1\}$. 

Take $\gamma = 1 \in \n$. The map $\phi : U \times \c^* \to \c^2$ is given by:
$$\phi((x: y : 1), t) = (t^{-1}x, t).$$ 

Take $\tau = z$. We have the following power series expansions:
\begin{align} 
f_1 &= x/z = u \\
\label{equ-ex1-f_i} 
f_2 &= y / z = 1 + \frac{u^3}{2} - \frac{u^6}{8} + \frac{u^9}{16} + \cdots \\
f_3 &= z / z = 1. 
\end{align}

The values of the valuation are: 
$$v(x/z)=1$$ $$v(y/z)=0$$ $$v(z/z)=0,$$ 
which clearly generate $\z$. We thus get:
\begin{align*}
\tilde{f}_1 &= \tilde{u} \\
\tilde{f}_2 &= 1 + \frac{t^3 \tilde{u}^3}{2} - \frac{t^6 \tilde{u}^6}{8} + \frac{t^9 \tilde{u}^9}{16} + \cdots \\
\tilde{f}_3 &= 1. 
\end{align*}

The embedding $F: \X \to \c\p^2 \times \c$ is given by:
\begin{align}
&F((x:y:z), t) = ((t^{-1}x:y:z), t), \quad t \neq 0,\\
&F(\tilde{u}, t) = ((\tilde{u} : 1 + \frac{t^3 \tilde{u}^3}{2} - \cdots : 1), t) \quad t \textup{ in a neighborhood of } 0. \label{equ-ex1-F}
\end{align}

At a point $(\tilde{u}, 0) \in \tilde{U} \subset \c^2$ the Jacobian of $F$ is:
$$JF(\tilde{u}, 0) = \left[ \begin{matrix} 1 & 0 & 0 \\ 0 & 0 & 1 \\ \end{matrix} \right],$$
which has rank $2$ as expected.

Let $\tilde{x} = t^{-1}x$, $\tilde{y}=y$ and $\tilde{z}=z$. The image $F(X \times \c^*) \subset \c\p^2 \times \c^*$ is given by:
$$F(X \times \c^*) = \{((\tilde{x}:\tilde{y}:\tilde{z}), t) \mid \tilde{z}\tilde{y}^2 = t^3 \tilde{x}^3 + \tilde{z}^3,~ t \neq 0 \}.$$ 
Let $\overline{\X}$ denote the closure of $F(X \times \c^*)$ in $\c\p^2 \times \c$. Consider $\overline{\X}$ as a family over $\c$. The fiber over $0$ of $\overline{\X}$ is given by the equation
$$\tilde{z}\tilde{y}^2 = \tilde{z}^3,$$
which is a union of the three lines $\tilde{y}=1$, $\tilde{y}=-1$ and $\tilde{z}=0$ in $\c\p^2$.
We note that by \eqref{equ-ex1-F}, for any $(\tilde{u}, t) \in \tilde{U}$, $\lim_{t \to 0} F(\tilde{u}, t) = ((\tilde{u} : 1 : 1), 0)$ which lies on the line $\tilde{y}=1$.
In other words, the family $\X$ can be realized as $\{((\tilde{x}:\tilde{y}:\tilde{z}), t) \mid \tilde{z}\tilde{y}^2 = t^3 \tilde{x}^3 + \tilde{z}^3,~ t \neq 0 \}$ union with the zero fiber 
$\{((\tilde{u}:1:1), 0) \mid \tilde{u} \in \c^* \} \cong \c^*$.

Notice that the degree of the zero fiber in $\X$ as a subvariety of $\c\p^2$ is $1$ because it is the projective line $\tilde{y}=1$ minus two points. On the other hand, 
the degree of the zero fiber of $\overline{\X}$ is $3$ because it is a union of the three lines $\tilde{y}=1$, $\tilde{y}=-1$ and $\tilde{z}=0$. Note that $3$ is the 
degree of the elliptic curve $X \subset \c\p^2$.
This agrees with the fact that degree of a fiber in a flat family does not change. 
\end{Ex} 

\begin{Ex} \label{ex-2}
In Example \ref{ex-1} instead of $\{x, y, z\}$ let us consider the basis $\{x, y-z, z\}$ for $(\c^3)^*$ and take the Hermitian product on $(\c^*)^3$ in which 
$\{x, y-z, z\}$ is orthonormal. To simplify the notation denote $y-z$ by $w$. With $p, u, \gamma$ and $\tau = z$ as before, from \eqref{equ-ex1-f_i} we get: 
$$v(w / z) = 3.$$
Let $G: \X \to \c\p^2 \times \c$ be the embedding obtained from the basis $\{x, w, z\}$. It is easy to see that the equation $zy^2 = x^3 + z^3$ can be rewritten 
as: $$zw^2 + 2z^2w = x^3.$$ 
If we put $\tilde{x} = t^{-1}x$, $\tilde{w} = t^{-3}w$ and $\tilde{z} = z$, the image $G(X \times \c^*) \subset \c\p^2 \times \c^*$ is given by the equation:
$$t^3\tilde{z} \tilde{w}^2 + 2\tilde{z}^2 \tilde{w} = \tilde{x}^3.$$
Let $\overline{\X}$ denote the closure of $G(X \times \c^*)$ in $\c\p^2 \times \c$.
Letting $t=0$ we get the curve:
$$2\tilde{z}^2 \tilde{w} = \tilde{x}^3.$$ 
This is an irreducible degree $3$ curve which is in fact a toric variety. It is the closure of the image of the monomial map:
$$ \tilde{u} \mapsto (\tilde{u} : \tilde{u}^3/2 : 1).$$
(Notice that this is different from the situation in 
Example \ref{ex-1} where the zero fiber in $\overline{\X}$ was a union of three degree $1$ curves.) 
This is an example where the zero fiber $X_0 \cong \c^*$ in the family $\X$ has  the same symplectic volume as the general fiber.
\end{Ex}

In Example \ref{ex-2},
the family $\overline{\X}$, i.e. the closure of the family $\X$ in $\c\p^2 \times \c$, is an example of the toric degenerations considered in \cite{Anderson, HK}.



\section{Enlarging the toric open subset} \label{sec-enlarge}
As before let $X$ be a smooth complex projective variety with a very ample line bundle $L$. We continue to follow the notation in Section \ref{sec-main}. In particular,
$E$ denotes the space $H^0(X, L)$ of holomorphic sections of $L$. The variety $X$ embeds in the projective space $\p(E^*)$ via the Kodaira map of $L$.
We fix a Hermitian product $H$ on $E^*$. 

For each integer $k>0$ we let 
$E^k$ denote the image of $E \otimes \cdots \otimes E$ ($k$ times) in $H^0(X, L^{\otimes k})$. The dual space $(E^k)^*$ can naturally be considered as a 
subspace of the $k$-th symmetric power $\Sym^k(E^*)$. We have the Hermitian product $H^{\otimes k}$ on $\Sym^k(E^*)$, and hence on $(E^k)^*$, defined by:
$$H^{\otimes k}(v_1 \cdot \ldots \cdot v_k, w_1 \cdot \ldots \cdot w_k) = H(v_1, w_1) \cdots H(v_k, w_k),$$
for $v_i, w_i \in E^*$. 

Let $v: \c(X) \setminus \{0\} \to \z^n$ be a valuation where $\z^n$ is equipped with a total order respecting addition (e.g. the valuation $v$ associated to a local coordinate system 
defined in Section \ref{sec-valuation}). 
We recall the definition of a Newtoon-Okounkov body associated to $(X, L, v)$ (see \cite{Okounkov, LM, KKh-Annals}). Fix a nonzero meromorphic section $\tau$.
\begin{Def} \label{def-NO-body}
The {\it Newton-Okounkov body} 
$\Delta = \Delta(X, L, v)$ is the convex body in $\r^n$ defined as:
$$\Delta = \overline{\conv(\bigcup_{k > 0} \{ \frac{1}{k} v(\sigma / \tau^k) \mid \sigma \in E^k \setminus \{0\} \})}.$$
\end{Def}
The dependence of $\Delta$ on $\tau$ is minor, a different choice of $\tau$ gives rise to a shifted convex body. On the other hand, the dependence of $\Delta$ on 
the valuation $v$ is very subtle.

The convex body $\Delta$ far generalizes the notion of Newton polytope of a projective toric variety (see \cite[Section 2]{CLS}
for the notion of Newton polytope of a projective toric variety). 

Note that $\Delta$ is not necessarily a polytope, although in many interesting examples it is indeed a rational polytope. 
The following is the main property of the convex body $\Delta$ (see \cite[Theorem 2.3]{LM} and \cite[Theorem 4.9]{KKh-Annals}). It is a generalization of the 
Bernstein-Kushnirenko theorem about degree of a projective toric variety and volume of its Newton polytope (Theorem \ref{th-BK}).
\begin{Th} \label{th-NO-body}
The convex body $\Delta$ has dimension $n$ and the degree of $(X, L)$, i.e. the self-intersection number of $c_1(L)$, is equal to $n!$ times the $n$-dimensional Euclidean volume of $\Delta$.
\end{Th}

\begin{Rem} \label{rem-NO-body}
Above we have only considered a very ample line bundle $L$.  
In the papers \cite{LM, KKh-Annals} the authors consider much more general situation of arbitrary graded linear systems on a variety $X$. They prove statements 
relating the asymptotic behavior of graded linear systems and the dimension and volume of their associated Newton-Okounkov bodies.
\end{Rem}

For each integer $k>0$ let $\A_k$ denote the set of values of the valuation $v$ on $E^k$. That is,
$$\A_k = \{ v(\sigma / \tau^k) \mid \sigma \in E^k \setminus \{0\} \}.$$
Moreover, put:
\begin{equation} \label{equ-Delta_k}
\Delta_k = \frac{1}{k} \conv(\A_k).
\end{equation}
Clearly, $\Delta_k \subset \r^n$ is a rational polytope because $k \Delta_k$ is a lattice polytope. Also for all $k>0$ we have $\Delta_k \subset \Delta$ and:
\begin{equation} \label{equ-Delta}
\Delta = \overline{\bigcup_{k>0} \Delta_k}.
\end{equation}
From \eqref{equ-Delta} we immediately obtain the following:
\begin{Prop} \label{prop-approx-Delta}
For any $\epsilon > 0$ there exists $N>0$ such that for any $k>N$ we have: $$\vol(\Delta \setminus \Delta_k) < \epsilon.$$
\end{Prop}

Take an integer $k>0$. Let us apply the constructions in Sections \ref{sec-valuation}
and \ref{sec-Kahler} to the embedding $X \hookrightarrow \p(\Sym^k(E^*))$. Here $X \hookrightarrow \p(\Sym^k(E^*))$ is the composition of the embedding
$X \hookrightarrow \p(E^*)$ with the $k$-th Veronese embedding $\p(E^*) \hookrightarrow \p(\Sym^k(E^*))$.

Let $r_k = \dim(E^k)$. Let $\{ \theta_1, \ldots, \theta_{r_k}\}$ be a basis such that the dual basis $\{ \theta_1^*, \ldots, \theta_{r_k}^*\}$ is an orthonormal basis for $(E^k)^*$
with respect to the Hermitian product $H^{\otimes k}$ and such that the values $v(\theta_j / \tau^k)$, $j=1, \ldots, r_k$, are all distinct. It follows that $\Delta_k$ is the convex hull of 
the $(1/k)v(\theta_j / \tau^k)$. 

{For each $j$ let us put $h_j = \theta_j / \tau^k$. 
As in Section \ref{sec-valuation} we choose $\gamma \in \n^n$ such that the constant terms in $t$ of the power series representations of the $\tilde{h}_j$, $j=1, \ldots, r_k$, are all monomials in $\tilde{u}$. Note that we chose $\gamma$ after the orthonormal basis $\{\theta_j\}$ was chosen.}


The construction in Section \ref{sec-Kahler} gives a K\"ahler form $\tilde{\omega}$ on the family $\X$ such that $\omega_1 = \tilde{\omega}_{|X_1}$ is the 
K\"ahler form $k\omega$ and $\omega_0 = \tilde{\omega}_{|X_0}$ is the K\"ahler form associated to the finite subset $\A_k$ and some $c_k \in (\c^*)^{r_k}$.

Recall that we call a K\"ahler form on $(\c^*)^n$ a rational toric K\"ahler form if it is of the form $(1/k) \omega_{\A_k, c_k}$ for some integer $k>0$, a finite subset
$\A_k \subset \z^n$ and $c_k \in (\c^*)^{r_k}$, where $r_k = |\A_k|$. We can now prove the main theorem of this section:
\begin{Th} \label{th-main2}
Let $X$ be a smooth projective variety with an integral K\"ahler form $\omega$ (i.e. the class of $\omega$ lies in $H^2(X, \z)$).
Let $\epsilon > 0$ be given. Then there exists an open subset $U \subset X$ (in the usual classical topology) such that:
\begin{itemize}
\item[(1)] The symplectic volume of $X \setminus U$ is less than $\epsilon$.
\item[(2)] $(U, \omega)$ is symplectomorphic to $(\c^*)^n$ equipped with a rational toric K\"ahler form.
\end{itemize}
\end{Th}
\begin{proof}
{Without loss of generality assume that $\omega$ represents the class of a very ample line bundle $L$.}
As in Proposition \ref{prop-approx-Delta} choose $k$ such that $\vol(\Delta \setminus \Delta_k) < \epsilon$ and consider the family $\X$ and the K\"ahler form $\tilde{\omega}$ as above. 
Let us equip $X_0 = (\c^*)^n$ with the K\"ahler form $(1/k)\omega_0$. Then $(\c^*)^n$ is an $(S^1)^n$-Hamiltonian manifold and the image of its moment map is 
the interior of the rational polytope $\Delta_k$. It follows that the symplectic volume of $((\c^*)^n, (1/k)\omega_0)$ is equal to $n! \vol(\Delta_k)$ (see Theorem \ref{th-BK}). 
By Theorem \ref{th-NO-body}, the symplectic volume of $(X, \omega)$ equals the Euclidean volume of $\Delta$. Thus,
the symplectic volume of $((\c^*)^n, (1/k)\omega_0) > $ the symplectic volume of $(X, \omega) - \epsilon$. Let $U \subset X$ be the open subset constructed in Theorem \ref{th-main}
such that $(U, k\omega)$ is symplectomorphic to $((\c^*)^n, \omega_0)$. Then the symplectic volume of $(U, \omega) = $ the symplectic volume of $((\c^*)^n, (1/k)\omega_0) > $ 
the symplectic volume of $(X, \omega) - \epsilon$. This finishes the proof.
\end{proof}

\begin{Rem} \label{rem-Delta-conv-A}
We know that, for any $k>0$, the polytope $\Delta_k = (1/k)\conv(\A)$ is contained in the Newton-Okounkov body $\Delta$. Indeed, if for some $k>0$ it happens 
that $\Delta = \Delta_k$ then $(X, \omega)$ has the same symplectic volume as the zero fiber $X_0$ equipped with the K\"ahler form $(1/k)\omega_{\A_k, c_k}$. 
It follows that the open subset $U$ is dense in $X$. This is the case considered in \cite{HK}.
Example \ref{ex-2} is an example of this (desirable) situation.
\end{Rem}

\section{A simplicial Newton-Okounkov body} \label{sec-simplex-NO-body}
Let $X$ be a complex projective variety of dimension $n$ with a very ample line bundle $L$.
In this section we recall a result from \cite{Seppanen} and \cite{AKL} which states that for an appropriate choice of a valuation $v$ (that depends on the line bundle $L$) 
the Newton-Okounkov body $\Delta$ is a simplex of a very specific form (Theorem \ref{th-NO-body-simplex}). 
 In Section \ref{sec-GW-packing} we will use this to prove a result about symplectic ball packings of $X$ 
(Corollary \ref{cor-symplectic-packing}).

As in Remark \ref{rem-diff-A-generates} let $\sigma_1, \ldots, \sigma_n \in E$ be sections of $L$ in general position such that the corresponding 
divisors $D_i = \textup{div}(\sigma_i)$ are smooth prime divisors intersecting transversely at $p$. Fix a section $\tau \in E$ such that $\tau(p) \neq 0$.
For each $i = 1, \ldots, n$ put:
$$u_i = {\frac{\sigma_i}{\tau}}.$$
Then the $u_i$ are rational function on $X$ that are regular at $p$ and form a local system of coordinates in a neighborhood of $p$.
Take the lexicographic order on $\z^n$ and let $v: \c(X) \setminus \{0\} \to \z^n$ be the initial term valuation as defined in \eqref{equ-valuation} in Section \ref{sec-valuation}.
The valuation $v$ can alternatively be defined as follows. For each $i=1, \ldots, n$ let $X_i$ denote the irreducible component of the intersection
$D_1 \cap \cdots \cap D_i$ containing $p$, in particular $X_n = \{p\}$ {(in fact if the $\sigma_i$ defining the $D_i$ are in general position then $X_1, \ldots, X_{n-1}$ are 
automatically irreducible).} Take $0 \neq f \in \c(X)$. Let $v_1 = \ord_{X_1}(f)$. Then $u_1^{-v_1}f$ is a rational function which does no have a zero or pole along the
hypersurface $X_1$. Let $f_1 = (u_1^{-v_1}f)_{|X_1}$ and $v_2 = \ord_{X_2}(f_1)$. Continue in the same manner to arrive at $(v_1, \ldots, v_n)$.
One verifies that in fact $v(f) = (v_1, \ldots, v_n)$.

We have the following. For completeness we include the nice short proof which is taken from \cite[Theorem 3.3]{Seppanen}.
\begin{Th} \label{th-NO-body-simplex}
Let the notation be as above.
\begin{itemize}
\item[(1)] The Newton-Okounkov body $\Delta = \Delta(X, L, v)$ is the simplex
$$\conv \{0, e_1, \ldots, e_{n-1}, de_n\},$$
where $d$ is the degree of the line bundle $L$, i.e. the degree of $X$ as a subvariety of $\p(E^*)$.
\item[(2)] Moreover, for any $\epsilon > 0$ one can find a sufficiently large integer $k$ such that the corresponding rational polytope 
$\Delta_k = (1/k)\conv(\A_k)$ contains a simplex:
$$\conv \{0, e_1, \ldots, e_{n-1}, d'e_n\},$$ 
for some $0 < d' \leq d$ with $|d-d'| < \epsilon$. Recall that $\A_k = \{ v(\sigma / \tau^k) \mid \sigma \in E^k \setminus \{0\}\}$ (see Section \ref{sec-enlarge}).
\end{itemize}
\end{Th}
\begin{proof}
First we note that for $i=1, \ldots, n$, we have $v(\sigma_i / \tau) = e_i$ where $e_i$ is the $i$-th standard basis element in $\r^n$. Also 
$v(\tau / \tau) = v(1) = 0$. Thus the convex body $\Delta$ contains the points $\{0, e_1, \ldots, e_n\}$.
One also observes that the degree of the line bundle $L_{n-1}$ on the curve $X_{n-1}$ is equal to $d$, the degree of $L$.
Let $v_{n-1}$ be the $\z$-valued valuation on $\c(X_{n-1})$ which is the order of zero-pole at the point $p$. Also let $\tau' = \tau_{|X_{n-1}}$.
Using Theorem \ref{th-NO-body} it is not difficult to see that the Newton-Okounkov body of $(X_{n-1}, L_{|X_{n-1}}, v_{n-1}, \tau')$ is the line segment $[0,d]$.
Thus for any $0 < d' < d$ we can find $m>0$ and a section $\sigma ' \in H^0(X_{n-1}, L_{|X_{n-1}}^{\otimes m})$ such that $d' < \frac{1}{m} v_{n-1}(\sigma'/\tau'^m) \leq d$.
Since $L$ is ample we can find $N_0 \in \n$ such that for all $j > N_0$ the restriction map
$$i_j: H^0(X, L^{\otimes j}) \to H^0(X_{n-1}, L^{\otimes j}_{|X_{n-1}}),$$ is surjective. Thus we conclude that for sufficiently large $N \in \n$ there is a 
section $\sigma \in H^0(X, L^{\otimes mN})$ such that
$i_{mN}(\sigma) = \sigma'^N$. Then $v(\sigma / \tau^{mN}) = (0, \ldots, 0, Nv_{n-1}(\sigma'/\tau'^m))$ and hence:
\begin{equation} \label{equ-d''}
\frac{1}{mN}v(\sigma/ \tau^{mN}) = (0, \ldots, 0, \frac{1}{m}v_{n-1}(\sigma'/\tau'^m)) \in \Delta.
\end{equation}
Let us write $d'' = \frac{1}{m} v_{n-1}(\sigma'/\tau'^m)$. The equation \eqref{equ-d''} tells us that the point $d''e_n$ lies in $\Delta$ and hence the 
simplices $\conv\{0, e_1, \ldots, e_{n-1}, d'e_n\} \subset \conv\{0, e_1, \ldots, e_{n-1}, d''e_n \}$ are contained in $\Delta$. 
This proves the claim in (2). To prove (1) we note that $\Delta$ is closed and $d'$ is arbitrary. It follows that 
the point $de_n = (0, \ldots, 0, d)$ lies in $\Delta$. This shows that the simplex $\conv\{0, e_1, \ldots, e_{n-1}, de_n\}$ is contained in 
$\Delta$. Finally, the volume of this simplex is equal to $d/n!$. On the other hand, by Theorem \ref{th-NO-body} 
the volume of $\Delta$ is also equal to $d/n!$. This finishes the proof of (1).
\end{proof}

\section{Lower bounds on Gromov width and symplectic ball packings} \label{sec-GW-packing}
In this section, as applications of results in Sections \ref{sec-main} and \ref{sec-enlarge}, we prove general statements about the Gromov width and 
symplectic ball packings of smooth projective varieties.

First we state a well-known fact about the Gromov width of a Hamiltonian $T$-space (see \cite{Traynor, Lu, Pabiniak}). 
Let $(M, \omega)$ be a (not necessarily compact) symplectic manifold of (real) dimension $2n$. Let us assume that we have an effective Hamiltonian action of the torus $T = (S^1)^n$
on $M$ (note that the dimension of $T$ is half of the dimension of $M$). In this case $M$ is usually called a toric symplectic manifold. Let us denote the moment map of 
$M$ by $\Phi: M \to \Lie(T)^* \cong \r^n$. The toric manifold $M$ is called {\it proper} if there exists an open and convex subset $\mathcal{T} \subset \Lie(T)^*$ 
containing the image $\Phi(M)$ of the moment map and such that $\Phi$ regarded as a map from $M$ to $\mathcal{T}$ is proper.

We will need the following definition.
\begin{Def} \label{def-size-simplex}
Let $\Delta(R)$ denote the (open) simplex $\{ (x_1, \ldots, x_n) \in \r_{> 0}^n \mid x_1 + \cdots + x_n < R\}$ in $\r^n$ of size $R$. 
We say that a simplex in $\r^n$ has {\it size} $R$ if it is obtained from $\Delta(R)$ after applying a linear transformation in $\textup{GL}(n, \z)$ and a translation in $\r^n$.
\end{Def}

\begin{Th}[Proposition 2.5 \cite{Pabiniak}] \label{th-Milena}
As above let $M$ be a toric manifold with a proper $T$-action. Suppose there is an (open) simplex of size $R$ that fits into the image of the moment map $\Phi(M)$.
{Then for any $\rho < R$ a ball of capacity $\rho$ (i.e. a ball with radius $\sqrt{\rho/ \pi}$) embeds symplectically into $(M, \omega)$, and thus the Gromov width of $M$ is at least $R$.} 
\end{Th}

The following is an immediate corollary of Theorem \ref{th-Milena}. As in Section \ref{sec-prelim} 
let $\A \subset \z^n$ be a finite subset and $c \in (\c^*)^r$ where $r = |\A|$. 
\begin{Cor} \label{cor-Gromov-width-torus}
The Gromov width of $((\c^*)^n, \omega_{\A, c})$ is at least $R$ 
where $R$ is the size of the largest (open) simplex that fits in the interior of  
$\conv(\A)$. In particular, the Gromov width of $((\c^*)^n, \omega_{\A, c})$ is at least $1$.
\end{Cor}

As usual let $X$ be a smooth projective variety of dimension $n$. Let $L$ be an ample line bundle on $X$ and $\omega$ a K\"ahler form on $X$ 
representing the Chern class $c_1(L)$. 
{Also let $v$ be the $\z^n$-valued valuation on the field of rational functions $\c(X)$ constructed in Section \ref{sec-valuation} and 
let $\Delta = \Delta(X, L, v)$ denote the associated Newton-Okounkov body.}
\begin{Cor} \label{cor-GW-NO-body}
The Gromov width of $(X, \omega)$ is at least $R$, where $R$ is {the supremum of the sizes of (open) simplices that fit in the 
interior of the convex body}
$\Delta = \Delta(X, L, v)$ (see Definition \ref{def-size-simplex}). In particular, if $L$ is very ample the Gromov width of $(X, \omega)$ is at least $1$.
\end{Cor}
\begin{proof}
{Let $\epsilon > 0$ be given. Then by assumption we can find  $\rho < R$ with $|R - \rho| < \epsilon$ such that the closure of 
a simplex $ a + W(\Delta(\rho))$, for some $W \in \textup{GL}(n, \z)$ and $a \in \r^n$, lies in the 
interior of $\Delta$. It follows from the definition of the Newton-Okounkov body $\Delta$ that 
we can find $k > 0$ such that the rational polytope $\Delta_k$ contains the simplex $a + W(\Delta(\rho))$. Recall that the polytope $\Delta_k$ is defined by:
$$\Delta_k = \frac{1}{k} \conv\{ v(\sigma / \tau^k) \mid \sigma \in E^k \setminus \{0\} \}.$$
Let $U_k$ denote the toric open subset in
Theorem \ref{th-main2} corresponding to the rational polytope $\Delta_k$. From Theorem 
\ref{th-Milena} we then conclude that a ball of capacity $\rho$ can be symplectically embedded into $U_k$. This implies that the Gromov width of $X$ is at least $R$. }
\end{proof}


\begin{Rem} \label{rem-Seshadri-constant}
{In algebraic geometry, given a projective variety $X$ and a line bundle $L$ one defines the notion of (very general) Seshadri constant $\epsilon(X, L)$ of $(X, L)$ 
(see \cite[Chapter 5]{Lazarsfeld-book}). 
It is a measure of 
local positivity of the line bundle $L$. In \cite{McDuff} it is proved that the Gromov width of $(X, L)$ is bounded below by the Seshadri constant $\epsilon(X, L)$. 
In \cite{Ito}, using different methods, Ito proves a bound on the Seshadri constant that is very close to that of Corollary \ref{cor-GW-NO-body}.} 
\end{Rem}

Recall that the degree $d$ of an ample line bundle $L$, i.e. the self-intersection of the divisor class of $L$,
is equal to $n!$ times the symplectic volume of $(X, \omega)$.
\begin{Cor} \label{cor-symplectic-packing}
Let us assume that $L$ is very ample. Then given any $\epsilon > 0$, the symplectic manifold $X$ has a symplectic ball packing by $d$ balls of capacity 
$1-\epsilon$ where $d$ is the degree of the line bundle $L$. In other words, $X$ has a full symplectic packing by $d$ equal balls.
\end{Cor}
\begin{proof}
Let $v$ be the initial term valuation described in Section \ref{sec-simplex-NO-body} which gives rise to the simplicial Newton-Okounkov body:
$$\Delta = \conv\{0, e_1, \ldots, e_{n-1}, de_n \}.$$
The simplex $\Delta$ can be written as the union of $d$ simplexes $\Delta_1, \ldots, \Delta_d$ where $$\Delta_i = \conv\{e_1, \ldots, e_{n-1}, (i-1)e_n, ie_n \}.$$ 
For each $i$ the simplex $\Delta_i$ is a lattice simplex and it is not difficult to see that $\Delta_i$ contains no lattice point in its interior. It follows that 
$\Delta_i$ can be transformed into the standard unit simplex $\{ (y_1, \ldots, y_n) \mid y_1 + \cdots + y_n < 1\}$ by a translation and a change of coordinates in 
$\textup{GL}(n, \z)$. Thus the $\Delta_i$ have size $1$ in the sense of Definition \ref{def-size-simplex}.
Take any $d''$ with $d-1 < d'' < d$. By Theorem \ref{th-NO-body-simplex}(2) we can find large enough integer $k>0$ 
such that the simplex: 
$$\Delta'' = \conv\{0, e_1, \ldots, e_{n-1}, d''e_n \},$$
is contained in $\Delta_k = (1/k) \conv(\A_k)$. Let $\Delta_{d''}$ denote the simplex:
$$\Delta_{d''} = \conv\{e_1, \ldots, e_{n-1}, (d-1)e_n, d''e_n \}.$$
Let $\epsilon > 0$ be given. By the proof of Theorem \ref{th-Milena} since the interiors of the simplices $\Delta_1, \ldots, \Delta_{n-1}$ and $\Delta_{d''}$ are disjoint, 
we can find symplectic embedding of a disjoint unions of $d$ balls of capacity $1-\epsilon$ into $X$ provided that $d''$ is sufficiently close to $d$.
This finishes the proof of the corollary.
\end{proof}


\end{document}